\documentclass[a4paper, 11pt]{article}
\usepackage[utf8]{inputenc}
\usepackage[english]{babel}
\usepackage{amssymb,amsmath,amsthm,amsfonts}

\usepackage{indentfirst}
\usepackage{comment}
\usepackage{enumerate}
\usepackage{enumitem}

\usepackage[pdftex]{graphicx}
\graphicspath{{pictures/}}
\DeclareGraphicsExtensions{.pdf,.png,.jpg}

\usepackage{geometry}
\newtheorem{theorem}{Theorem}
\newtheorem{lemma}[theorem]{Lemma}

\newtheorem{remark}{Remark}

\DeclareMathOperator*{\argmin}{argmin}
\DeclareMathOperator*{\prox}{prox}

\DeclareMathOperator*{\proj}{proj}
\DeclareMathOperator*{\Id}{Id}

\DeclareMathOperator*{\gra}{gra}

\usepackage{float} 
\usepackage{caption}
\usepackage{subcaption}

\usepackage[usenames]{color}
\usepackage{colortbl}

\begin{document}

\title{A fast continuous time approach for non-smooth convex optimization using Tikhonov regularization technique}

\author{Mikhail A. Karapetyants 
\footnote{Faculty of Mathematics, University of Vienna, Oskar-Morgenstern-Platz 1, 1090 Vienna, Austria, 
{email: \tt mikhail.karapetyants@univie.ac.at.} Research supported by the Doctoral Programme \emph{Vienna Graduate School on Computational Optimization (VGSCO)} which is funded by FWF (Austrian Science Fund), project W 1260.} 
}

\maketitle

\begin{abstract}

In this manuscript we would like to address the classical optimization problem of minimizing a proper, convex and lower semicontinuous function via the second order in time dynamics, combining viscous and Hessian-driven damping with a Tikhonov regularization technique. In our analysis we heavily exploit the Moreau envelope of the objective function and its properties as well as Tikhonov properties, which we extend to a nonsmooth case. We introduce the setting, which at the same time guarantees the fast convergence of the function (and Moreau envelope) values and strong convergence of the trajectories of the system to a minimal norm solution -- the element of the minimal norm of all the minimizers of the objective. Moreover, we deduce the precise rates of convergence of the values for the particular choice of parameter function. Various numerical examples are also included as an illustration of the theoretical results.



\end{abstract} 

\section{Introduction}

In the Hilbert framework $H$, $\langle \cdot, \cdot \rangle$, $ \| \cdot \| = \sqrt{\langle \cdot, \cdot \rangle} $, we study the convergence properties of the second order in time differential equation
\begin{equation}\label{Syst_0}
    \ddot x(t) + \alpha \sqrt{\varepsilon(t)} \dot x(t) + \beta \frac{d}{dt} \left( \nabla \Phi_{\lambda(t)} (x(t)) \right) + \nabla \Phi_{\lambda(t)} (x(t)) + \varepsilon(t) x(t) = 0 \text{ for } t \geq t_0,
\end{equation}
where the initial conditions are $x(t_0) = x_0 \in H$ and $\dot x(t_0) = x_1 \in H$ and $ \alpha, \beta \text{ and } t_0 > 0 $, $ \Phi: H \longrightarrow \overline{\mathbb{R}} = \mathbb{R} \cup \{ \pm \infty \} $ is a proper, convex and lower semicontinuous function and $\Phi_\lambda$ is its Moreau envelope of the index $\lambda > 0$. The function $\lambda: [t_0, +\infty) \longrightarrow \mathbb{R}_+$ is assumed to be continuously differentiable and nondecreasing while the function $\varepsilon: [t_0, +\infty) \longrightarrow \mathbb{R}_+$ is continuously differentiable and nonincreasing with the property $\lim_{t \to +\infty} \varepsilon(t) = 0$. In addition, we assume that $\argmin \Phi$, which is the set of global minimizers of $\Phi$, is not empty and denote by $\Phi^*$ the optimal objective value of $\Phi$. Finally, for every $t \geq t_0$ let us introduce the strongly convex function $\varphi_{\varepsilon(t), \lambda(t)}: H \longrightarrow \mathbb{R}$ is defined as $\varphi_{\varepsilon(t), \lambda(t)} (x) = \Phi_{\lambda(t)}(x) + \frac{\varepsilon(t) \| x \|^2}{2}$. Let us denote the unique minimizer of $\varphi_{\varepsilon(t), \lambda(t)}$ as $x_{\varepsilon(t), \lambda(t)} = \argmin_{H} \varphi_{\varepsilon(t), \lambda(t)}$. The system \eqref{Syst_0} has a connection to the minimization problem
\[
\min_{x \in H} \Phi(x)
\]
of a proper, convex and lower semicontinuous function $\Phi$.

The main goal of the research is to provide the setting in a nonsmooth case, where we would have fast convergence of the function values combined with strong convergence of the trajectories -- the solution of \eqref{Syst_0} -- to the element of the minimal norm from the set of all minimizers of the objective function. This analysis is an extrapolation of the one conducted in \cite{ABCR_0} to the case of a nonsmooth objective function. We also aim to provide the exact rates of convergence of the values for the polynomial choice of the smoothing parameter $\lambda$ and Tikhonov function $\varepsilon$. As a conclusion, multiple numerical experiments were conducted allowing better understanding of the theoretical results.

\subsection{Nonsmooth optimization}

The Moreau envelope plays a significant role in nonsmooth optimization. It is defined as ($\Phi : H \to \overline{\mathbb{R}}$ is a proper, convex and lower semicontinuous function) 
\begin{equation*}
\Phi_\lambda: H \to \mathbb{R}, \quad \Phi_{\lambda} (x) \ = \ \inf_{y \in H} \left\{ \Phi(y) + \frac{1}{2 \lambda} \| x - y \|^2 \right\},
\end{equation*}
where $\lambda > 0$. $\Phi_\lambda$ is convex and continuously differentiable with
\begin{align}\label{Morprox}
\nabla \Phi_{\lambda} (x) = \frac{1}{\lambda} ( x - \prox\nolimits_{\lambda \Phi} (x)) \quad \forall x \in H,
\end{align}
and $\nabla \Phi_\lambda$ is $\frac{1}{\lambda}$-Lipschitz continuous. Here,
\[
\prox\nolimits_{\lambda \Phi}: H \to H, \quad \prox\nolimits_{\lambda \Phi} (x) = \argmin_{y \in H} \left\{ \Phi(y) + \frac{1}{2 \lambda} \| x - y \|^2 \right\}
\]
denotes the proximal operator of $\Phi$ of parameter $\lambda$. Moreover (see \cite{AC}),
\begin{equation}\label{deriv}
    \frac{d}{dt} \Phi_{\lambda(t)} (x) = - \frac{\dot \lambda(t)}{2} \| \nabla \Phi_{\lambda(t)} (x) \|^2 \ \forall x \in H.
\end{equation}

The work \cite{AL} by Attouch-L\'aszl\'o serves as a starting point for a lot of different research topics in nonsmooth optimization. The following dynamics was considered
\begin{equation}\label{NSystHess}
\ddot x(t) + \frac{\alpha}{t} \dot x(t) + \beta \frac{d}{dt} \nabla \Phi_{\lambda(t)}(x(t)) + \nabla \Phi_{\lambda(t)}(x(t)) = 0
\end{equation}
where $\alpha > 1$ and $\beta > 0$, and the term $\frac{d}{dt} \nabla \Phi_{\lambda(t)}(x(t))$ is inspired by the Hessian driven damping term in the case of smooth functions. For this system multiple fundamental results were proven, such as convergence rates for the Moreau envelope values as well as for the velocity of the system 
\[
\Phi_{\lambda(t)}(x(t)) - \Phi^* = o\left( \frac{1}{t^2} \right) \text{ and } \| \dot x(t) \| = o\left( \frac{1}{t} \right) \text{ as } t \to +\infty,
\]
from where convergence rates for the  $\Phi$ along the trajectories themselves were deduced
\[
\Phi \big( \prox\nolimits_{\lambda(t) \Phi} (x(t)) \big) - \Phi^* = o\left( \frac{1}{t^2} \right) \text{ and } \| \prox\nolimits_{\lambda(t) \Phi} (x(t)) - x(t) \| = o\left( \frac{\sqrt{\lambda(t)}}{t} \right) \text{ as } t \to +\infty.
\]
In addition, convergence rates for the gradient of the Moreau envelope of parameter $\lambda(t)$ and its time derivative along $x(t)$ were established
\[
\| \nabla \Phi_{\lambda(t)} (x(t)) \| = o\left( \frac{1}{t^2} \right) \text{ and } \left\| \frac{d}{dt} \nabla \Phi_{\lambda(t)} (x(t)) \right\| = o\left( \frac{1}{t^2} \right) \text{ as } t \to +\infty.
\]
Moreover, the weak convergence of the trajectories $x(t)$ to a minimizer of $\Phi$ as $t \to +\infty$ was deduced. 

From here one may go in many directions in order to continue investigating the topic of second order dynamics. Time scaling, for instance, can be introduced to improve the speed of convergence of the values, as it was done in \cite{BK}. Another way to proceed is to consider the so-called Tikhonov regularization technique, to which we devote the next chapter of our manuscript.

\subsection{Tikhonov regularization}

The presence of the Tikhonov term in the system equation dramatically influences the behaviour of its trajectories, namely, under some appropriate conditions, it improves the convergence of the trajectories from weak to a strong one. Not only that, but it also ensures the convergence not to an arbitrary element from the set of all minimizers of the objective, but to the particular one, which has the smallest norm. Under the presence of the Tikhonov term in the system it is still possible to obtain fast rates of convergece of the function values. Systems with Tikhonov regularization were studied in, for instance, in \cite{ABCR, ABCR_0, ACR, AL_0, BCL, BCL_0, CK, L}. 

One of the fine examples in a smooth setting is presented below (see \cite{ABCR_0})
\begin{equation*}
    \ddot x(t) + \alpha \sqrt{\varepsilon(t)} \dot x(t) + \beta \frac{d}{dt} \Big( \nabla \varphi_t (x(t)) + (p - 1) \varepsilon(t) x(t) \Big) + \nabla \varphi_t (x(t)) = 0 \text{ for } t \geq t_0,
\end{equation*}
where $\varphi_t (x) = \Phi(x) + \frac{\varepsilon(t) \| x \|^2}{2}$, $ \Phi: H \longrightarrow \mathbb{R}$ is twice continuously differentiable and convex, $\varepsilon$ is nonincreasing and goes to zero, as $t \to +\infty$, and $p$ is chosen appropriately. This system inherits the properties of fast convergence rates of the function values, being of the order $\frac{1}{t^2}$, and additionally provides the strong convergence results for the trajectories of the system in the same setting.

Concerning the nonsmooth case we refer to \cite{BCL_0}, where it was covered for the more general systems, governed by a maximally monotone operator, but with a different damping. The authors studied the following dynamics
\begin{equation}\label{OS}
    \ddot x(t) + \frac{\alpha}{t^q} \dot x(t) + \beta \frac{d}{dt} \Big( A_{\lambda(t)} (x(t)) \Big) + A_{\lambda(t)} (x(t)) + \varepsilon(t) x(t) = 0 \text{ for } t \geq t_0,
\end{equation}
where $\alpha > 0$, $\beta \geq 0$, $0 < q \leq 1$ and $\lambda(t) = \lambda t^{2q}$ for $\lambda > 0$, $A$ is a maximally monotone operator and $A_\lambda$ is its Yosida regularization of the order $\lambda$. The system \eqref{OS} is related to the inclusion problem $0 \in Ax$. The authors showed the fast convergence rates for $\| \dot x(t) \|$, $\| A_{\lambda(t)} (x(t)) \|$ and $\| \frac{d}{dt} A_{\lambda(t)} (x(t)) \|$ being of the order $\frac{1}{t^q}$, $\frac{1}{t^{2q}}$ and $\frac{1}{t^{3q}}$ correspondingly. Moreover, they established the strong convergence of the trajectories of the system.

\begin{remark}

We would like to stress that Theorem $11$ of \cite{BCL_0} does not cover the case presented in this paper.

\begin{enumerate}
    
    \item First of all, the systems \eqref{Syst_0} and \eqref{OS} have different damping coefficients. The damping in \eqref{Syst_0} depends on the Tikhonov function $\varepsilon$, while the damping in \eqref{OS} is taken in a polynomial form $\frac{1}{t^q}$. Thus, if we take $\varepsilon(t) = \frac{1}{t^{2q}}$ in \eqref{OS} to mimic the relation between the damping parameter and the Tikhonov function as in \eqref{Syst_0}, then one of the conditions of Theorem $11$ becomes
    \[
    \int_{t_0}^{+\infty} t^{3q} \varepsilon^2(t) dt \ = \ \int_{t_0}^{+\infty} \frac{1}{t^q} dt \ < \ +\infty,
    \]
    where $0 < q < 1$, which is obviously not fulfilled.

    \item Secondly, the smoothing parameter $\lambda$ in \cite{BCL_0} is fixed, while our analysis holds for more general choice of $\lambda$. However, if we want to consider the polynomial case of parameters (Section $5$), then we indeed arrive at a similar restriction for $\lambda$: in Section $6$ we will discover that for strong conergence of the trajectories and polynomial choice of parameters, $\lambda(t) = t^l$, we have to take $0 \leq l < 2$, which is a wider range than $0 < q < 1$ for $\lambda(t) = t^{2q}$.

    \item Finally, some of the conditions $\left( C_0 \right)$ -- $\left( C_4 \right)$ actually contradict with some of our assumptions \eqref{assumption_0} -- \eqref{assumption_3}.

\end{enumerate}

\end{remark}

In this paper we aim to develop the ideas presented in \cite{ABCR_0} for $p = 0$ to cover the nonsmooth case. Section $2$ gathers some preliminary results, which we will need in our analysis. The main result of our research is presented in Section $3$. In Section $4$ we study 
in more detalies the results of the previous section in order to show that they are meaningful. Section $5$ provides the polynomial setting, in which the results are valid and the analysis works. Section $6$ establishes the actual rates of convergence of the values and the trajectories. In Section $7$ we consider an interesting particular case $\varepsilon(t) = \frac{1}{t^2}$ and show fast convergence of the function values. Finally, Section $8$ is all about numerical experiments, which illustrate the theory.

\section{Preparatory results}

\subsection{Auxiliary estimates and properties}

Let us begin with two important properties, which we will later use in our analysis. The first one concerns the proximal mapping:
\begin{equation}\label{proxineq}
    \|\prox\nolimits_{\lambda \Phi}(x) - \prox\nolimits_{\mu \Phi}(x) \| \leq |\lambda  - \mu| \|\nabla \Phi_{\lambda} (x)\| \ \forall \lambda, \ \mu > 0.
\end{equation}
The second one, which is known as the first order optimality condition, in our case reads as
\begin{equation}\label{FOOC}
    \nabla \Phi_{\lambda(t)} (x_{\varepsilon(t), \lambda(t)}) + \varepsilon(t) x_{\varepsilon(t), \lambda(t)} = 0.
\end{equation}

We continue with the following lemma (see \cite{BC}, Proposition 12.22, for the first term of the lemma and \cite{AP}, Appendix, A$1$, for the second one). 

\begin{lemma}\label{L_0}

Let $\Phi: H \longrightarrow \overline{\mathbb{R}}$ be a proper, convex and lower semicontinuous function, $\lambda, \mu > 0$. Then
\begin{enumerate}

    \item\label{i} $(\Phi_\lambda)_\mu = \Phi_{\lambda + \mu}$.
    
    \item\label{ii} $ \prox_{\mu \Phi_\lambda} = \frac{\lambda}{\lambda + \mu} \Id + \frac{\mu}{\lambda + \mu} \prox_{(\lambda + \mu) \Phi} $.
    
\end{enumerate}

\end{lemma}

The following estimates will be used later to evaluate the derivative of our energy function. 

\begin{lemma}\label{L_1}

The following properties are satisfied:

\begin{enumerate}

    \item\label{i_0} for each $t \geq t_0$, $\frac{d}{dt} \left( \varphi_{\varepsilon(t), \lambda(t)} (x_{\varepsilon(t), \lambda(t)}) \right) = \frac{1}{2} \left( \dot \varepsilon(t) - \dot \lambda(t) \varepsilon^2(t) \right) \| x_{\varepsilon(t), \lambda(t)} \|^2$;
    
    \item\label{ii_0} the function $t \mapsto x_{\varepsilon(t), \lambda(t)}$ is Lipschitz continuous on the compact intervals of $(t_0, +\infty)$, thus, is almost everywhere differentiable. Moreover, for almost every $t \geq t_0$
    \[
    \left( \frac{2 \dot \lambda(t)}{\lambda(t)} - \frac{\dot \varepsilon(t)}{\varepsilon(t)} \right) \| x_{\varepsilon(t), \lambda(t)} \| \geq \left\| \frac{d}{dt} x_{\varepsilon(t), \lambda(t)} \right\|.
    \]

\end{enumerate}

\end{lemma}

\begin{proof}

By the definition of $\varphi_{\varepsilon(t), \lambda(t)}$ 
\[
\varphi_{\varepsilon(t), \lambda(t)} (x_{\varepsilon(t), \lambda(t)}) = \inf_{y \in H} \left( \Phi_{\lambda(t)} (y) + \frac{\varepsilon(t)}{2} \| y - 0 \|^2 \right) = \left( \Phi_{\lambda(t)} \right)_{\frac{1}{\varepsilon(t)}} (0) = \Phi_{\lambda(t) + \frac{1}{\varepsilon(t)}} (0)
\]
by \ref{i} of Lemma \ref{L_0}. Thus,
\[
\frac{d}{dt} \left( \varphi_{\varepsilon(t), \lambda(t)} (x_{\varepsilon(t), \lambda(t)}) \right) \ = \ \frac{d}{dt} \left( \Phi_{\lambda(t) + \frac{1}{\varepsilon(t)}} (0) \right) = \frac{1}{2} \left( \frac{\dot \varepsilon(t)}{\varepsilon^2(t)} - \dot \lambda(t) \right) \| \nabla \Phi_{\lambda(t) + \frac{1}{\varepsilon(t)}} (0) \|^2,
\]
by \eqref{deriv}. From \eqref{FOOC} we obtain
\[
x_{\varepsilon(t), \lambda(t)} = \prox\nolimits_{\frac{1}{\varepsilon(t)}\Phi_{\lambda(t)}} (0) = \frac{\prox\nolimits_{\left( \lambda(t) + \frac{1}{\varepsilon(t)} \right) \Phi} (0)}{\lambda(t) \varepsilon(t) + 1},
\]
where the second equality comes from \ref{ii} of Lemma \ref{L_0}. Combining the last two equalities with \eqref{Morprox} we obtain
\begin{align*}
    \frac{d}{dt} \left( \varphi_{\varepsilon(t), \lambda(t)} (x_{\varepsilon(t), \lambda(t)}) \right) \ &= \ \frac{1}{2} \left( \frac{\dot \varepsilon(t)}{\varepsilon^2(t)} - \dot \lambda(t) \right) \left \| \frac{ 0 - \prox\nolimits_{\left( \lambda(t) + \frac{1}{\varepsilon(t)} \right) \Phi} (0)}{\lambda(t) + \frac{1}{\varepsilon(t)}} \right \|^2 \\
    &= \ \frac{1}{2} \left( \frac{\dot \varepsilon(t)}{\varepsilon^2(t)} - \dot \lambda(t) \right) \left \| \frac{ - \left( \lambda(t) \varepsilon(t) + 1 \right) x_{\varepsilon(t), \lambda(t)}}{\lambda(t) + \frac{1}{\varepsilon(t)}} \right \|^2 \\
    &= \ \frac{1}{2} \left( \dot \varepsilon(t) - \dot \lambda(t) \varepsilon^2(t) \right) \| x_{\varepsilon(t), \lambda(t)} \|^2,
\end{align*}
which is the first claim.

To obtain the second claim we start with \eqref{FOOC} noticing that for $h > 0$
\[
\nabla \Phi_{\lambda(t)} \left( x_{\varepsilon(t), \lambda(t)} \right) = -\varepsilon(t) x_{\varepsilon(t), \lambda(t)} \text{ and } \nabla \Phi_{\lambda(t+h)} (x_{\varepsilon(t+h), \lambda(t+h)}) = -\varepsilon(t+h) x_{\varepsilon(t+h), \lambda(t+h)}.
\]
Consider
\begin{align*}
    &\nabla \Phi_{\lambda(t+h)} (x_{\varepsilon(t+h), \lambda(t+h)}) - \nabla \Phi_{\lambda(t)} (x_{\varepsilon(t), \lambda(t)}) \\
    = \ &\nabla \Phi_{\lambda(t+h)} (x_{\varepsilon(t+h), \lambda(t+h)}) - \nabla \Phi_{\lambda(t+h)} (x_{\varepsilon(t), \lambda(t)}) + \nabla \Phi_{\lambda(t+h)} (x_{\varepsilon(t), \lambda(t)}) - \nabla \Phi_{\lambda(t)} (x_{\varepsilon(t), \lambda(t)}).
\end{align*}
Taking the inner product of each part of this equality with $x_{\varepsilon(t+h), \lambda(t+h)} - x_{\varepsilon(t), \lambda(t)}$, we notice that
\[
\langle \nabla \Phi_{\lambda(t+h)} (x_{\varepsilon(t+h), \lambda(t+h)}) - \nabla \Phi_{\lambda(t+h)} (x_{\varepsilon(t), \lambda(t)}), x_{\varepsilon(t+h), \lambda(t+h)} - x_{\varepsilon(t), \lambda(t)} \rangle \geq 0
\]
by the monotonicity of $\nabla \Phi_{\lambda(t+h)}$. So,
\begin{align*}
    &\langle \varepsilon(t) x_{\varepsilon(t), \lambda(t)} - \varepsilon(t+h) x_{\varepsilon(t+h), \lambda(t+h)}, x_{\varepsilon(t+h), \lambda(t+h)} - x_{\varepsilon(t), \lambda(t)} \rangle \ = \ -\varepsilon(t) \| x_{\varepsilon(t+h), \lambda(t+h)} - x_{\varepsilon(t), \lambda(t)} \|^2 \\
    &+ \ \left( \varepsilon(t) - \varepsilon(t+h) \right) \langle x_{\varepsilon(t+h), \lambda(t+h)}, x_{\varepsilon(t+h), \lambda(t+h)} - x_{\varepsilon(t), \lambda(t)} \rangle \\
    \geq \ &\langle \nabla \Phi_{\lambda(t+h)} (x_{\varepsilon(t), \lambda(t)}) - \nabla \Phi_{\lambda(t)} (x_{\varepsilon(t), \lambda(t)}), x_{\varepsilon(t+h), \lambda(t+h)} - x_{\varepsilon(t), \lambda(t)} \rangle.
\end{align*}
Let us divide the last inequality by $h^2$ to obtain
\begin{align*}
&\frac{\varepsilon(t) - \varepsilon(t+h)}{h} \left\langle x_{\varepsilon(t+h), \lambda(t+h)}, \frac{x_{\varepsilon(t+h), \lambda(t+h)} - x_{\varepsilon(t), \lambda(t)}}{h} \right\rangle \ \geq \ \varepsilon(t) \left\| \frac{x_{\varepsilon(t+h), \lambda(t+h)} - x_{\varepsilon(t), \lambda(t)}}{h} \right\|^2 \\
&+ \ \left\langle \frac{\nabla \Phi_{\lambda(t+h)} (x_{\varepsilon(t), \lambda(t)}) - \nabla \Phi_{\lambda(t)} (x_{\varepsilon(t), \lambda(t)})}{h}, \frac{x_{\varepsilon(t+h), \lambda(t+h)} - x_{\varepsilon(t), \lambda(t)}}{h} \right\rangle.
\end{align*}
Now notice that, since the mapping $t \mapsto x_{\varepsilon(t), \lambda(t)}$ is Lipschitz continuous on the compact intervals of $\mathbb{R}_+ \setminus \{0\}$ (according to \cite{AC}), therefore, almost everywhere differentiable. Tending $h$ to zero we deduce for almost every $t \geq t_0$
\[
-\dot \varepsilon(t) \left\langle x_{\varepsilon(t), \lambda(t)}, \frac{d}{dt} x_{\varepsilon(t), \lambda(t)} \right\rangle \ \geq \ \varepsilon(t) \left\| \frac{d}{dt} x_{\varepsilon(t), \lambda(t)} \right\|^2 - \frac{2\dot \lambda(t) \| \nabla \Phi_{\lambda(t)} (x_{\varepsilon(t), \lambda(t)}) \| \left\| \frac{d}{dt} x_{\varepsilon(t), \lambda(t)} \right\|}{\lambda(t)},
\]
where we used the following estimate from \cite{BK}
\begin{align*}
    &\lim_{h \to 0} \left \langle \frac{\nabla \Phi_{\lambda(t + h)} (x_{\varepsilon(t), \lambda(t)}) - \nabla \Phi_{\lambda(t)} (x_{\varepsilon(t), \lambda(t)})}{h}, \frac{x_{\varepsilon(t+h), \lambda(t+h)} - x_{\varepsilon(t), \lambda(t)}}{h} \right \rangle \\ 
    = \ &\lim_{h \to 0} \left \langle \frac{\nabla \Phi_{\lambda(t + h)} (x_{\varepsilon(t), \lambda(t)}) - \nabla \Phi_{\lambda(t)} (x_{\varepsilon(t), \lambda(t)})}{h}, \frac{x_{\varepsilon(t+h), \lambda(t+h)} - x_{\varepsilon(t), \lambda(t)}}{h} \right \rangle \\
    = \ &\lim_{h \to 0} \Bigg \langle \frac{ (\lambda(t + h)) \prox_{\lambda(t) \Phi} (x_{\varepsilon(t), \lambda(t)}) - \lambda(t) \prox_{(\lambda(t + h)) \Phi} (x_{\varepsilon(t), \lambda(t)}) - \big( \lambda(t + h) - \lambda(t) \big) x_{\varepsilon(t), \lambda(t)} }{\lambda(t) \lambda(t + h)h}, \\
    &\frac{x_{\varepsilon(t+h), \lambda(t+h)} - x_{\varepsilon(t), \lambda(t)}}{h} \Bigg \rangle \\
    = \ &\lim_{h \to 0} \left \langle \frac{\big( \lambda(t + h) - \lambda(t) \big) \big( \prox_{\lambda(t) \Phi} (x_{\varepsilon(t), \lambda(t)}) - x_{\varepsilon(t), \lambda(t)} \big)}{\lambda(t) \lambda(t + h) h}, \frac{x_{\varepsilon(t+h), \lambda(t+h)} - x_{\varepsilon(t), \lambda(t)}}{h} \right \rangle \\
    &- \ \lim_{h \to 0} \left \langle \frac{\prox_{(\lambda(t + h)) \Phi} (x_{\varepsilon(t), \lambda(t)}) - \prox_{\lambda(t) \Phi} (x_{\varepsilon(t), \lambda(t)})}{\lambda(t+h)h}, \frac{x_{\varepsilon(t+h), \lambda(t+h)} - x_{\varepsilon(t), \lambda(t)}}{h} \right \rangle \\
    \geq \ &\frac{\dot \lambda(t)}{\lambda^2(t)} \left \langle \prox\nolimits_{\lambda(t) \Phi} (x_{\varepsilon(t), \lambda(t)}) - x_{\varepsilon(t), \lambda(t)}, \frac{d}{dt} x_{\varepsilon(t), \lambda(t)} \right \rangle \\
    &- \ \lim_{h \to 0} \frac{(\lambda(t + h) - \lambda(t)) \| \nabla \Phi_{\lambda(t)} (x_{\varepsilon(t), \lambda(t)}) \| \left\| \frac{x_{\varepsilon(t+h), \lambda(t+h)} - x_{\varepsilon(t), \lambda(t)}}{h} \right\| }{\lambda(t + h) h} \\
    = \ &\frac{\dot \lambda(t)}{\lambda^2(t)} \left \langle \prox\nolimits_{\lambda(t) \Phi} (x_{\varepsilon(t), \lambda(t)}) - x_{\varepsilon(t), \lambda(t)}, \frac{d}{dt} x_{\varepsilon(t), \lambda(t)} \right \rangle \ - \ \frac{\dot \lambda(t) \| \nabla \Phi_{\lambda(t)} (x_{\varepsilon(t), \lambda(t)}) \| \left\| \frac{d}{dt} x_{\varepsilon(t), \lambda(t)} \right\|}{\lambda(t)} \\
    = \ & -\frac{\dot \lambda(t)}{\lambda(t)} \left \langle \nabla \Phi_{\lambda(t)} (x_{\varepsilon(t), \lambda(t)}), \frac{d}{dt} x_{\varepsilon(t), \lambda(t)} \right \rangle - \ \frac{\dot \lambda(t) \| \nabla \Phi_{\lambda(t)} (x_{\varepsilon(t), \lambda(t)}) \| \left\| \frac{d}{dt} x_{\varepsilon(t), \lambda(t)} \right\|}{\lambda(t)} \\ \geq \ & - \frac{2\dot \lambda(t) \| \nabla \Phi_{\lambda(t)} (x_{\varepsilon(t), \lambda(t)}) \| \left\| \frac{d}{dt} x_{\varepsilon(t), \lambda(t)} \right\|}{\lambda(t)},
\end{align*}
where we used \eqref{Morprox}, \eqref{proxineq} and the Cauchy-Schwarz inequality. On the other hand, Cauchy-Schwartz inequality yields
\[
-\dot \varepsilon(t) \left\langle x_{\varepsilon(t), \lambda(t)}, \frac{d}{dt} x_{\varepsilon(t), \lambda(t)} \right\rangle \leq -\dot \varepsilon(t) \| x_{\varepsilon(t), \lambda(t)} \| \left\| \frac{d}{dt} x_{\varepsilon(t), \lambda(t)} \right\|.
\]
Combining the last two inequalities we arrive at
\[
-\dot \varepsilon(t) \| x_{\varepsilon(t), \lambda(t)} \| + \frac{2\dot \lambda(t)}{\lambda(t)} \left\| \nabla \Phi_{\lambda(t)} (x_{\varepsilon(t), \lambda(t)}) \right\| \ \geq \ \varepsilon(t) \left\| \frac{d}{dt} x_{\varepsilon(t), \lambda(t)} \right\|.
\]
Replacing $\nabla \Phi_{\lambda(t)} (x_{\varepsilon(t), \lambda(t)})$ using \eqref{FOOC} gives us the second claim.

\end{proof}

Let us also mention two key properties of the Tikhonov regularization, which we will use later in the analysis

\begin{lemma}\label{Tikhonov}

Suppose that 
\begin{equation}\label{A_0}
    \lim_{t \to +\infty} \lambda(t) \varepsilon(t) \ = \ 0.
\end{equation}
Then the following properties of the mapping $t \longrightarrow x_{\varepsilon(t), \lambda(t)}$ are satisfied:
\begin{equation}\label{T_0}
    \text{ for } x^* = \proj\nolimits_{\argmin \Phi}(0), \ \| x_{\varepsilon(t), \lambda(t)} \| \leq \| x^* \| \text{ for all } t \geq t_0
\end{equation}
and
\begin{equation}\label{t_0}
    \lim_{t \to +\infty} \| x_{\varepsilon(t), \lambda(t)} - x^* \| = 0.
\end{equation}

\end{lemma}

\begin{proof}

Suppose that $t \geq t_0$. By the monotonicity of $\nabla \Phi_{\lambda(t)}$ we deduce
\[
\left\langle \nabla \Phi_{\lambda(t)}(x_{\varepsilon(t), \lambda(t)}) - \nabla \Phi_{\lambda(t)}(x^*), x_{\varepsilon(t), \lambda(t)} - x^* \right\rangle \ \geq \ 0 \text{ for all } t \geq t_0.
\]
By \eqref{FOOC} we obtain
\[
\left\langle - \varepsilon(t) x_{\varepsilon(t), \lambda(t)}, x_{\varepsilon(t), \lambda(t)} - x^* \right\rangle \ = \ \varepsilon(t) \left( - \| x_{\varepsilon(t), \lambda(t)} \|^2 + \left\langle x_{\varepsilon(t), \lambda(t)}, x^* \right\rangle \right) \ \geq \ 0.
\]
Using Cauchy-Schwarz inequality we derive
\[
\| x_{\varepsilon(t), \lambda(t)} \| \ \leq \ \| x^* \|.
\]
This proves the first claim. For the second one consider \eqref{FOOC} again and note that it is equivalent to 
\begin{equation}\label{Char}
    x_{\varepsilon(t), \lambda(t)} = \prox\nolimits_{\frac{1}{\varepsilon(t)}\Phi_{\lambda(t)}} (0) = \frac{\prox\nolimits_{\left( \lambda(t) + \frac{1}{\varepsilon(t)} \right) \Phi} (0)}{\lambda(t) \varepsilon(t) + 1}
\end{equation}
by the item \ref{ii} of Lemma \ref{L_0}. Note that by \eqref{A_0} we have $\lambda(t) + \frac{1}{\varepsilon(t)} \to +\infty$ and $\lambda(t) \varepsilon(t) + 1 \to 1$, as $t \to +\infty$. From now on the proof is inspired by Theorem 23.44 of \cite{BC}. Take $z \in \argmin \Phi = \argmin \Phi_\lambda$ for each $\lambda > 0$. From \eqref{Char} and from the fact that the resolvent of maximally monotone operator is maximally monotone and firmly nonexpansive (see, for instance, Corollary 23.11(i) of \cite{BC}) and Cauchy-Schwarz inequality it follows that for all $t \geq t_0$ (note that $z$ could be represented as $z = \prox\nolimits_{\frac{1}{\varepsilon(t)}\Phi_{\lambda(t)}} (z)$)
\begin{equation}\label{Ineq}
    \left\| z - x_{\varepsilon(t), \lambda(t)} \right\| \| z - 0 \| \ \geq \ \left\langle z - x_{\varepsilon(t), \lambda(t)}, z - 0 \right\rangle \ \geq \ \| z - x_{\varepsilon(t), \lambda(t)} \|^2,
\end{equation}
which gives the boundedness of $x_{\varepsilon(t), \lambda(t)}$ for all $t \geq t_0$. Now, let $y$ be a weak sequential cluster point of $\{ x_{\varepsilon(t_n), \lambda(t_n)} \}_{n \in \mathbb{N}}$, namely, $x_{\varepsilon(t_{k_n}), \lambda(t_{k_n})} \rightharpoonup y$, as $n \to +\infty$. From \eqref{FOOC} we deduce
\[
\nabla \Phi_{\lambda(t_{k_n})} (x_{\varepsilon(t_{k_n}), \lambda(t_{k_n})}) + \varepsilon(t_{k_n}) x_{\varepsilon(t_{k_n}), \lambda(t_{k_n})} = 0.
\]
Using
\[
\nabla \Phi_\lambda \ = \ \left( \partial \Phi \right)_\lambda \ = \ \frac{1}{\lambda} \left( \Id - \left( \Id + \lambda \partial \Phi \right)^{-1} \right) \quad \forall \lambda > 0
\]
we further obtain
\[
\left( \Id + \lambda(t_{k_n}) \partial \Phi \right)^{-1} \big( x_{\varepsilon(t_{k_n}), \lambda(t_{k_n})} \big) \ = \ \Big( \lambda(t_{k_n}) \varepsilon(t_{k_n}) + 1 \Big) x_{\varepsilon(t_{k_n}), \lambda(t_{k_n})},
\]
which is equivalent to
\[
x_{\varepsilon(t_{k_n}), \lambda(t_{k_n})} \ \in \ \Big( \lambda(t_{k_n}) \varepsilon(t_{k_n}) + 1 \Big) x_{\varepsilon(t_{k_n}), \lambda(t_{k_n})} + \lambda(t_{k_n}) \partial \Phi \left( \left( \lambda(t_{k_n}) \varepsilon(t_{k_n}) + 1 \right) x_{\varepsilon(t_{k_n}), \lambda(t_{k_n})} \right)
\]
or
\begin{equation}\label{FOOC_new}
    \frac{ -\varepsilon(t_{k_n}) }{\lambda(t_{k_n}) \varepsilon(t_{k_n}) + 1} x_{\varepsilon(t_{k_n}), \lambda(t_{k_n})} \ \in \ \partial \Phi \left( x_{\varepsilon(t_{k_n}), \lambda(t_{k_n})} \right).
\end{equation}
The sequence 
\[
\left\{ x_{\varepsilon(t_{k_n}), \lambda(t_{k_n})}, \frac{ -\varepsilon(t_{k_n}) }{\lambda(t_{k_n}) \varepsilon(t_{k_n}) + 1} x_{\varepsilon(t_{k_n}), \lambda(t_{k_n})} \right\}
\]
lies in $\gra \partial \Phi$ by \eqref{FOOC_new} and converges to $(y, 0)$ in $H^{weak} \times H^{strong}$ due to the sequence $\{ x_{\varepsilon(t_{k_n}), \lambda(t_{k_n})} \}_{n \in \mathbb{N}}$ being also bounded and \eqref{A_0}. Therefore, since $\gra \partial \Phi$ is sequentially closed (see Proposition 20.38(ii) of \cite{BC}) it follows that $y \in \argmin \Phi$. From \eqref{Ineq} we derive 
\[
\| y - x_{\varepsilon(t_{k_n}), \lambda(t_{k_n})} \|^2 \ \leq \ \left\langle y - 0, y - x_{\varepsilon(t_{k_n}), \lambda(t_{k_n})} \right\rangle \to 0, \text{ as } n \to +\infty, 
\]
by the definition of weak convergence, thus, $x_{\varepsilon(t_{k_n}), \lambda(t_{k_n})} \to y$, as $n \to +\infty$. On the other hand, \eqref{Ineq} leads to
\begin{align*}
    0 \ &\geq \ \| z - x_{\varepsilon(t_{k_n}), \lambda(t_{k_n})} \|^2 - \left\langle z - x_{\varepsilon(t_{k_n}), \lambda(t_{k_n})}, z - 0 \right\rangle \\
    &= \ \left\langle z - x_{\varepsilon(t_{k_n}), \lambda(t_{k_n})}, 0 - x_{\varepsilon(t_{k_n}), \lambda(t_{k_n})} \right\rangle \to \left\langle z - y, 0 - y \right\rangle
\end{align*}
and thus $y = x^*$ by the characterization of $x^*$, namely,
\[
\text{for} x^* \in \argmin \Phi \text{ and } \forall z \in \argmin \Phi \text{ it holds that } \langle z - x^*, 0 - x^* \rangle \leq 0.
\] 
So, $x^*$ being the only weak sequential cluster point of the bounded sequence $\left\{ x_{\varepsilon(t_n), \lambda(t_n)} \right\}_{n \in \mathbb{N}}$ means that $x_{\varepsilon(t_n), \lambda(t_n)} \rightharpoonup x^*$, as $n \to +\infty$, by Lemma 2.46 of \cite{BC}. By \eqref{Ineq} again we deduce
\[
\| x_{\varepsilon(t_n), \lambda(t_n)} - x^* \|^2 \ \leq \ \left\langle 0 - x^*, x_{\varepsilon(t_n), \lambda(t_n)} - x^* \right\rangle \to 0, \text{ as } n \to +\infty
\]
and so the second claim follows.

\end{proof}

\subsection{Existence and uniqueness of the solution of \eqref{Syst_0}}

Our nearest goal is to deduce the existence and uniqueness of the solution of the dynamical system \eqref{Syst_0}. Suppose $\beta > 0$. Let us integrate \eqref{Syst_0} from $t_0$ to $t$ to obtain
\[
    \dot x(t) + \beta \nabla \Phi_{\lambda(t)} (x(t)) + \int_{t_0}^t \left( \alpha \sqrt{\varepsilon(s)} \dot x(s) + \nabla \Phi_{\lambda(s)} (x(s)) + \varepsilon(s) x(s) \right) ds - \dot x(t_0) - \beta \nabla \Phi_{\lambda(t_0)} (x(t_0)) \ = \ 0.
\]
Denoting $z(t) := \int_{t_0}^t \left( \alpha \sqrt{\varepsilon(s)}  \dot x(s) + \nabla \Phi_{\lambda(s)} (x(s)) + \varepsilon(s) x(s) \right) ds - \big( \dot x(t_0) + \beta \nabla \Phi_{\lambda(t_0)} (x_0)) \big)$ for every $t \geq t_0$ and noticing that $\dot z(t) = \alpha \sqrt{\varepsilon(t)} \dot x(t) + \nabla \Phi_{\lambda(t)} (x(t)) + \varepsilon(t) x(t)$ we deduce, that \eqref{Syst_0} is equivalent  to
\begin{equation*}
\begin{cases}
    &\dot x(t) + \beta \nabla \Phi_{\lambda(t)} (x(t)) + z(t) = 0, \\
    &\dot z(t) - \alpha \sqrt{\varepsilon(t)} \dot x(t) - \nabla \Phi_{\lambda(t)}(x(t)) - \varepsilon(t) x(t) = 0, \\
    &x(t_0) = x_0, \ z(t_0) = -\left(\dot x(t_0) + \beta \nabla \Phi_{\lambda(t_0)} (x_0) \right).
\end{cases}
\end{equation*}
Let us multiply the second one by $\beta$ and then by summing it with the first line we get rid of the gradient of the Moreau envelope in the second equation
\begin{equation*}
\begin{cases}
    &\dot x(t) + \beta \nabla \Phi_{\lambda(t)} (x(t)) + z(t) = 0, \\
    &\beta \dot z(t) + \left( 1 - \alpha \beta \sqrt{\varepsilon(t)} \right) \dot x(t) - \beta \varepsilon(t) x(t) + z(t) = 0,\\
    &x(t_0) = x_0, \ z(t_0) = -\left(\dot x(t_0) + \beta \nabla \Phi_{\lambda(t_0)} (x_0) \right).
\end{cases}
\end{equation*}
We denote now $y(t) = \beta z(t) + \left( 1 - \alpha \beta \sqrt{\varepsilon(t)} \right) x(t)$, and, after simplification, we obtain the following equivalent formulation for the dynamical system 
\begin{equation*}
\begin{cases}
    &\dot x(t) + \beta \nabla \Phi_{\lambda(t)} (x(t)) + \left( \alpha \sqrt{\varepsilon(t)} - \frac{1}{\beta} \right) x(t) + \frac{1}{\beta} y(t) = 0, \\
    &\dot y(t) + \left( \frac{\alpha \beta \dot \varepsilon(t)}{2\sqrt{\varepsilon(t)}} - \beta \varepsilon(t) - \frac{1}{\beta} + \alpha \sqrt{\varepsilon(t)} \right) x(t) + \frac{1}{\beta} y(t) = 0, \\
    &x(t_0) = x_0, \ y(t_0) = -\beta \left( \dot x(t_0) + \beta \nabla \Phi_{\lambda(t_0)} (x_0) \right) + \left( 1 - \alpha \beta \sqrt{\varepsilon(t_0)} \right) x_0.
\end{cases}
\end{equation*}
In case $\beta = 0$ for every $t \geq t_0$, \eqref{Syst_0}  can be equivalently written as
\begin{equation*}
\begin{cases}
&\dot x(t)  - y(t) = 0, \\
&\dot y(t) + \alpha \sqrt{\varepsilon(t)} y(t) + \nabla \Phi_{\lambda(t)}(x(t)) + \varepsilon(t) x(t) = 0, \\
&x(t_0) = x_0, \ y(t_0) = \dot x(t_0).
\end{cases}
\end{equation*}
Therefore, based on the two reformulations of the dynamical system \eqref{Syst_0} above we provide the following existence and uniqueness result, which is a consequence of the Cauchy-Lipschitz theorem for strong global solutions. The proof follows the lines of the proofs of Theorem $1$ in \cite{AL} or of Theorem $1.1$ in \cite{APR} with some small adjustments.

\begin{theorem}

Suppose that there exists $\lambda_0 > 0$ such that $\lambda(t) \geq \lambda_0$ for all $t \geq t_0$. Then for every $(x_0, \dot x(t_0)) \in H \times H $ there exists a unique strong global solution $x: [t_0, +\infty) \mapsto H$ of the continuous dynamics \eqref{Syst_0} which satisfies the Cauchy initial conditions $x(t_0) = x_0$ and $\dot x(t_0) = \dot x_0$.

\end{theorem}

\section{Main result}

This section is devoted to establishing some crucial estimates for the following quantities \\ $\Phi_{\lambda(t)}(x(t)) - \Phi^*$ and $\|x(t) - x_{\varepsilon(t), \lambda(t)}\|$ for all $t \geq t_0$. In order to do so we will use the ideas and methods of Lyapunov analysis. We introduce the energy function
\begin{equation}\label{Energy_0}
\begin{split}
    E(t) \ = \ &\varphi_{\varepsilon(t), \lambda(t)} (x(t)) - \varphi_{\varepsilon(t), \lambda(t)} (x_{\varepsilon(t), \lambda(t)}) \\
    &+ \ \frac{1}{2} \left \| \gamma \sqrt{\varepsilon(t)} \left( x(t) - x_{\varepsilon(t), \lambda(t)} \right) + \dot x(t) + \beta \nabla \Phi_{\lambda(t)} (x(t)) \right \|^2,
\end{split}
\end{equation}
where $\frac{\alpha}{2} \leq \gamma < \alpha$.

The idea is to show that this energy function satisfies the following differential inequality, as it was done in \cite{ABCR_0},
\[
\dot E(t) + \mu(t) E(t) + \frac{\beta}{2} \| \nabla \varphi_{\varepsilon(t), \lambda(t)} (x(t)) \|^2 \ \leq \ \frac{g(t) \| x^* \|^2}{2} \text{ for all } t \geq t_0,
\]
where $\mu(t) = \left( \alpha - \gamma \right) \sqrt{\varepsilon(t)} - \frac{\dot \varepsilon(t)}{2 \varepsilon(t)}$ and $g$ are positive functions. The next theorem provides the analysis needed to obtain the desired inequality.

\begin{theorem}\label{Main}

Let $x: [t_0, +\infty) \longrightarrow H$ be a solution of \eqref{Syst_0}. Assume that \eqref{A_0} holds and suppose that there exist $a, c > 0$ such that for $t$ large enough it holds that
\begin{equation}\label{assumption_0}
    \frac{d}{dt} \left( \frac{1}{\sqrt{\varepsilon(t)}} \right) \ \leq \ \min \left\{ 2 \gamma - \alpha - \frac{\gamma \beta \dot \varepsilon(t)}{2 \varepsilon(t)}, \ \alpha - \gamma \frac{a + 1}{a} \right\}  
\end{equation}
\begin{equation}\label{assumption_1}
    \left( 2 \gamma (\alpha - \gamma)  + \frac{\gamma}{c} - 1 \right) \varepsilon(t) - \beta \dot \varepsilon(t) \ \leq \ 0,        
\end{equation}
\begin{equation}\label{assumption_2}
    2 \beta \varepsilon^2(t) + \left( 2 - \gamma \beta \sqrt{\varepsilon(t)} \right) \dot \varepsilon(t) \ \leq \ 0
\end{equation}
and
\begin{equation}\label{assumption_3}
    \left( \frac{\gamma}{a} + 2 (\alpha - \gamma) \right) \beta^2 \sqrt{\varepsilon(t)} - \frac{3 \beta^2 \dot \varepsilon(t)}{2 \varepsilon(t)} - \dot \lambda(t) \ \leq \ \beta.
\end{equation}
Then there exists $t_1 \geq t_0$ such that for all $t \geq t_1$
\[
\beta \int_{t_1}^t \| \nabla \varphi_{\varepsilon(s)} (x(s)) \|^2 ds \ \leq \ 2 E(t_1) + \| x^* \|^2 \int_{t_1}^t g(s) ds
\]
and
\[
E(t) \ \leq \ \frac{\| x^* \|^2}{2 \Gamma(t)} \int_{t_1}^t \Gamma(s) g(s) ds + \frac{\Gamma(t_1) E(t_1)}{\Gamma(t)},
\]
where $\Gamma(t) = \exp \left( \int_{t_1}^t \mu(s) ds \right)$ and $g(t) = \dot \lambda(t) \varepsilon^2(t) - \dot \varepsilon(t) + \frac{\gamma \beta \dot \varepsilon(t) \sqrt{\varepsilon(t)}}{2} + \gamma (2a + c \gamma) \sqrt{\varepsilon(t)} \left( \frac{2 \dot \lambda(t)}{\lambda(t)} - \frac{\dot \varepsilon(t)}{\varepsilon(t)} \right)^2 $.

\end{theorem}

\begin{proof}

We start with computing the derivative of the energy function \eqref{Energy_0}. Let us denote $v(t) = \gamma \sqrt{\varepsilon(t)} \left( x(t) - x_{\varepsilon(t), \lambda(t)} \right) + \dot x(t) + \beta \nabla \Phi_{\lambda(t)} (x(t))$. Once again, by the classical derivation chain rule using \eqref{i} from Lemma \ref{L_1} and \eqref{deriv} we obtain for all $t \geq t_0$
\begin{align*}
    \dot E(t) \ = \ &\langle \nabla \varphi_{\varepsilon(t), \lambda(t)} (x(t)), \dot x(t) \rangle + \frac{\dot \varepsilon(t)}{2} \| x(t) \|^2 + \frac{1}{2} \left( \dot \lambda(t) \varepsilon^2(t) - \dot \varepsilon(t) \right) \| x_{\varepsilon(t), \lambda(t)} \|^2 \\
    &+ \ \langle \dot v(t), v(t) \rangle - \frac{\dot \lambda(t)}{2} \| \nabla \Phi_{\lambda(t)} (x(t)) \|^2.
\end{align*}
Our nearest goal is to obtain the upper bound for $\dot E$. Let us calculate for all $t \geq t_0$
\begin{align*}
    \dot v(t) \ &= \ \frac{\gamma \dot \varepsilon(t)}{2 \sqrt{\varepsilon(t)}} \left( x(t) - x_{\varepsilon(t), \lambda(t)} \right) + \gamma \sqrt{\varepsilon(t)} \dot x(t) - \gamma \sqrt{\varepsilon(t)} \frac{d}{dt} x_{\varepsilon(t), \lambda(t)} + \ddot x(t) + \beta \frac{d}{dt} \left( \nabla \Phi_{\lambda(t)} (x(t)) \right) \\
    &= \ \frac{\gamma \dot \varepsilon(t)}{2 \sqrt{\varepsilon(t)}} \left( x(t) - x_{\varepsilon(t), \lambda(t)} \right) + \left( \gamma - \alpha \right) \sqrt{\varepsilon(t)} \dot x(t) - \gamma \sqrt{\varepsilon(t)} \frac{d}{dt} x_{\varepsilon(t), \lambda(t)} - \nabla \Phi_{\lambda(t)} (x(t)) - \varepsilon(t) x(t) \\
    &= \ \frac{\gamma \dot \varepsilon(t)}{2 \sqrt{\varepsilon(t)}} \left( x(t) - x_{\varepsilon(t), \lambda(t)} \right) + \left( \gamma - \alpha \right) \sqrt{\varepsilon(t)} \dot x(t) - \gamma \sqrt{\varepsilon(t)} \frac{d}{dt} x_{\varepsilon(t), \lambda(t)} - \nabla \varphi_{\varepsilon(t), \lambda(t)} (x(t)),
\end{align*}
where above we used \eqref{Syst_0}. Thus, for all $t \geq t_0$
\begin{align*}
    \langle \dot v(t), v(t) \rangle \ = \ &\frac{\gamma^2 \dot \varepsilon(t)}{2} \| x(t) - x_{\varepsilon(t), \lambda(t)} \|^2 + \left( \frac{\gamma \dot \varepsilon(t)}{2 \sqrt{\varepsilon(t)}} + \gamma (\gamma - \alpha) \varepsilon(t) \right) \langle x(t) - x_{\varepsilon(t), \lambda(t)}, \dot x(t) \rangle \\
    &- \ \gamma^2 \varepsilon(t) \left \langle \frac{d}{dt} x_{\varepsilon(t), \lambda(t)}, x(t) - x_{\varepsilon(t), \lambda(t)} \right \rangle - \gamma \sqrt{\varepsilon(t)} \left \langle x(t) - x_{\varepsilon(t), \lambda(t)}, \nabla \varphi_{\varepsilon(t), \lambda(t)} (x(t)) \right \rangle \\
    &+ \ \left( \gamma - \alpha \right) \sqrt{\varepsilon(t)} \| \dot x(t) \|^2 - \gamma \sqrt{\varepsilon(t)} \left \langle \frac{d}{dt} x_{\varepsilon(t), \lambda(t)}, \dot x(t) \right \rangle - \left \langle \nabla \varphi_{\varepsilon(t), \lambda(t)} (x(t)), \dot x(t) \right \rangle \\
    &+ \ \frac{\gamma \beta \dot \varepsilon(t)}{2 \sqrt{\varepsilon(t)}} \left \langle x(t) - x_{\varepsilon(t), \lambda(t)}, \nabla \Phi_{\lambda(t)} (x(t)) \right \rangle + \beta \left( \gamma - \alpha \right) \sqrt{\varepsilon(t)} \left \langle \nabla \Phi_{\lambda(t)} (x(t)), \dot x(t) \right \rangle \\ 
    &- \ \gamma \beta \sqrt{\varepsilon(t)} \left \langle \frac{d}{dt} x_{\varepsilon(t), \lambda(t)}, \nabla \Phi_{\lambda(t)} (x(t)) \right \rangle - \beta \left \langle \nabla \varphi_{\varepsilon(t), \lambda(t)} (x(t)), \nabla \Phi_{\lambda(t)} (x(t)) \right \rangle.
\end{align*}
Let us use the previous estimates to evaluate the quantity $\langle \dot v(t), v(t) \rangle$. Namely, by the $\varepsilon(t)$-strong convexity of $\varphi_{\varepsilon(t), \lambda(t)}$ for all $t \geq t_0$
\[
\varphi_{\varepsilon(t), \lambda(t)} (x_{\varepsilon(t), \lambda(t)}) - \varphi_{\varepsilon(t), \lambda(t)} (x(t)) \ \geq \ \left\langle \nabla \varphi_{\varepsilon(t), \lambda(t)} (x(t)), x_{\varepsilon(t), \lambda(t)} - x(t) \right\rangle + \frac{\varepsilon(t)}{2} \| x_{\varepsilon(t), \lambda(t)} - x(t) \|^2
\]
and then for all $t \geq t_0$
\begin{align*}
    &- \gamma \sqrt{\varepsilon(t)} \left \langle x(t) - x_{\varepsilon(t), \lambda(t)}, \nabla \varphi_{\varepsilon(t), \lambda(t)} (x(t)) \right \rangle \\
    \leq \ &- \gamma \sqrt{\varepsilon(t)} \left( \varphi_{\varepsilon(t), \lambda(t)} (x(t)) - \varphi_{\varepsilon(t), \lambda(t)} (x_{\varepsilon(t), \lambda(t)}) \right) - \frac{\gamma \varepsilon^{\frac{3}{2}}(t)}{2} \| x_{\varepsilon(t), \lambda(t)} - x(t) \|^2.
\end{align*}
Again, by the $\varepsilon(t)$-strong convexity of $\varphi_{\varepsilon(t), \lambda(t)}$ since $\dot \varepsilon(t) \ \leq \ 0$ for all $t \geq t_0$ 
\begin{align*}
    &\frac{\gamma \beta \dot \varepsilon(t)}{2 \sqrt{\varepsilon(t)}} \left \langle x(t) - x_{\varepsilon(t), \lambda(t)}, \nabla \Phi_{\lambda(t)} (x(t)) + \varepsilon(t) x(t) - \varepsilon(t) x(t) \right \rangle \\
    \leq \ &\frac{\gamma \beta \dot \varepsilon(t)}{2 \sqrt{\varepsilon(t)}} \Bigg( \varphi_{\varepsilon(t), \lambda(t)} (x(t)) - \varphi_{\varepsilon(t), \lambda(t)} (x_{\varepsilon(t), \lambda(t)}) - \frac{\varepsilon(t)}{2} \| x(t) - x_{\varepsilon(t), \lambda(t)} \|^2 \\
    &- \ \varepsilon(t) \left\langle x(t) - x_{\varepsilon(t), \lambda(t)}, x(t) \right\rangle \Bigg).
\end{align*}
Furthermore,
\[
- \varepsilon(t) \left\langle x(t) - x_{\varepsilon(t), \lambda(t)}, x(t) \right\rangle \ = \ -\frac{\varepsilon(t)}{2} \Big( \| x(t) - x_{\varepsilon(t), \lambda(t)} \|^2 + \| x(t) \|^2 - \| x_{\varepsilon(t), \lambda(t)} \|^2 \Big).
\]
It is true that for all $a > 0$
\[
- \gamma \sqrt{\varepsilon(t)} \left \langle \frac{d}{dt} x_{\varepsilon(t), \lambda(t)}, \dot x(t) \right \rangle \ \leq \ \frac{\gamma \sqrt{\varepsilon(t)}}{2a} \| \dot x(t) \|^2 + \frac{a \gamma \sqrt{\varepsilon(t)}}{2} \left \| \frac{d}{dt} x_{\varepsilon(t), \lambda(t)} \right \|^2
\]
as well as
\[
- \gamma \beta \sqrt{\varepsilon(t)} \left \langle \frac{d}{dt} x_{\varepsilon(t), \lambda(t)}, \nabla \Phi_{\lambda(t)} (x(t)) \right \rangle \ \leq \ \frac{\gamma \beta^2 \sqrt{\varepsilon(t)}}{2a} \| \nabla \Phi_{\lambda(t)} (x(t)) \|^2 + \frac{a \gamma \sqrt{\varepsilon(t)}}{2} \left \| \frac{d}{dt} x_{\varepsilon(t), \lambda(t)} \right \|^2.
\]
In the same spirit for all $b > 0$
\[
- \gamma^2 \varepsilon(t) \left \langle \frac{d}{dt} x_{\varepsilon(t), \lambda(t)}, x(t) - x_{\varepsilon(t), \lambda(t)} \right \rangle \ \leq \ \frac{b \gamma \sqrt{\varepsilon(t)}}{2} \left \| \frac{d}{dt} x_{\varepsilon(t), \lambda(t)} \right \|^2 + \frac{\gamma^3 \varepsilon^{\frac{3}{2}}(t)}{2b} \| x(t) - x_{\varepsilon(t), \lambda(t)} \|^2.
\]
Furthermore,
\begin{align*}
    \left( \gamma - \alpha \right) \sqrt{\varepsilon(t)} \left( \| \dot x(t) \|^2 + \beta \left \langle \nabla \Phi_{\lambda(t)} (x(t)), \dot x(t) \right \rangle \right) \ = \ &\frac{\left( \gamma - \alpha \right) \sqrt{\varepsilon(t)}}{2} \Big( \| \dot x(t) \|^2 + \| \dot x(t) + \beta \nabla \Phi_{\lambda(t)} (x(t)) \|^2 \\
    &- \ \beta^2 \| \nabla \Phi_{\lambda(t)} (x(t)) \|^2 \Big)
\end{align*}
and
\begin{equation}\label{estimate_0}
\begin{split}
    &- \beta \left \langle \nabla \varphi_{\varepsilon(t), \lambda(t)} (x(t)), \nabla \Phi_{\lambda(t)} (x(t)) \right \rangle \\ 
    = \ &- \frac{\beta}{2} \Big( \| \nabla \varphi_{\varepsilon(t), \lambda(t)} (x(t)) \|^2 + \| \nabla \Phi_{\lambda(t)} (x(t)) \|^2 - \| \nabla \varphi_{\varepsilon(t), \lambda(t)} (x(t)) - \nabla \Phi_{\lambda(t)} (x(t)) \|^2 \Big) \\
    = \ &- \frac{\beta}{2} \Big( \| \nabla \varphi_{\varepsilon(t), \lambda(t)} (x(t)) \|^2 + \| \nabla \Phi_{\lambda(t)} (x(t)) \|^2 - \varepsilon^2(t) \| x(t) \|^2 \Big)
\end{split}
\end{equation}
Combining all the estimates above we arrive for all $t \geq t_0$ at
\begin{align*}
    \langle \dot v(t), v(t) \rangle \ \leq \ &\left(\frac{\gamma \beta \dot \varepsilon(t)}{2 \sqrt{\varepsilon(t)}} - \gamma \sqrt{\varepsilon(t)} \right) \left( \varphi_{\varepsilon(t), \lambda(t)} (x(t)) - \varphi_{\varepsilon(t), \lambda(t)} (x_{\varepsilon(t), \lambda(t)}) \right) \\
    &+ \ \left( \frac{\gamma \dot \varepsilon(t)}{2 \sqrt{\varepsilon(t)}} + \gamma (\gamma - \alpha) \varepsilon(t) \right) \langle x(t) - x_{\varepsilon(t), \lambda(t)}, \dot x(t) \rangle \\
    &+ \ \left( \frac{\gamma^2 \dot \varepsilon(t)}{2} + \frac{\gamma^3 \varepsilon^{\frac{3}{2}}(t)}{2b} - \frac{\gamma \varepsilon^{\frac{3}{2}}(t)}{2} - \frac{\gamma \beta \dot \varepsilon(t) \sqrt{\varepsilon(t)}}{2} \right) \| x(t) - x_{\varepsilon(t), \lambda(t)} \|^2 \\
    &+ \ \left( \frac{\gamma}{a} + \gamma - \alpha \right) \frac{\sqrt{\varepsilon(t)}}{2} \| \dot x(t) \|^2 + \frac{(\gamma - \alpha) \sqrt{\varepsilon(t)}}{2} \| \dot x(t) + \beta \nabla \Phi_{\lambda(t)} (x(t)) \|^2 \\
    &+ \ \frac{\gamma (2a + b) \sqrt{\varepsilon(t)}}{2} \left \| \frac{d}{dt} x_{\varepsilon(t), \lambda(t)} \right \|^2 \\
    &+ \ \frac{1}{2} \left( \frac{\gamma \beta^2 \sqrt{\varepsilon(t)}}{a} - \beta - \beta^2 (\gamma - \alpha) \sqrt{\varepsilon(t)} \right) \| \nabla \Phi_{\lambda(t)} (x(t)) \|^2 \\
    &- \ \left \langle \nabla \varphi_{\varepsilon(t), \lambda(t)} (x(t)), \dot x(t) \right \rangle + \left( \frac{\beta \varepsilon^2(t)}{2} - \frac{\gamma \beta \dot \varepsilon(t) \sqrt{\varepsilon(t)}}{4} \right) \| x(t) \|^2 - \frac{\beta}{2} \| \nabla \varphi_{\varepsilon(t), \lambda(t)} (x(t)) \|^2 \\
    &+ \ \frac{\gamma \beta \dot \varepsilon(t) \sqrt{\varepsilon(t)}}{4} \| x_{\varepsilon(t), \lambda(t)} \|^2.
\end{align*}
Returning to the expression for $\dot E(t)$ we notice that the terms $\left \langle \nabla \varphi_{\varepsilon(t), \lambda(t)} (x(t)), \dot x(t) \right \rangle$ cancel each other out.
\begin{align*}
    \dot E(t) \ \leq \ &\left(\frac{\gamma \beta \dot \varepsilon(t)}{2 \sqrt{\varepsilon(t)}} - \gamma \sqrt{\varepsilon(t)} \right) \left( \varphi_{\varepsilon(t), \lambda(t)} (x(t)) - \varphi_{\varepsilon(t), \lambda(t)} (x_{\varepsilon(t), \lambda(t)}) \right) \\
    &+ \ \left( \frac{\gamma \dot \varepsilon(t)}{2 \sqrt{\varepsilon(t)}} + \gamma (\gamma - \alpha) \varepsilon(t) \right) \langle x(t) - x_{\varepsilon(t), \lambda(t)}, \dot x(t) \rangle \\
    &+ \ \left( \frac{\gamma^2 \dot \varepsilon(t)}{2} + \frac{\gamma^3 \varepsilon^{\frac{3}{2}}(t)}{2b} - \frac{\gamma \varepsilon^{\frac{3}{2}}(t)}{2} - \frac{\gamma \beta \dot \varepsilon(t) \sqrt{\varepsilon(t)}}{2} \right) \| x(t) - x_{\varepsilon(t), \lambda(t)} \|^2 \\
    &+ \ \left( \frac{\gamma}{a} + \gamma - \alpha \right) \frac{\sqrt{\varepsilon(t)}}{2} \| \dot x(t) \|^2 + \frac{(\gamma - \alpha) \sqrt{\varepsilon(t)}}{2} \| \dot x(t) + \beta \nabla \Phi_{\lambda(t)} (x(t)) \|^2 \\
    &+ \ \frac{\gamma (2a + b) \sqrt{\varepsilon(t)}}{2} \left \| \frac{d}{dt} x_{\varepsilon(t), \lambda(t)} \right \|^2 \\
    &+ \ \frac{1}{2} \left( \frac{\gamma \beta^2 \sqrt{\varepsilon(t)}}{a} - \beta - \beta^2 (\gamma - \alpha) \sqrt{\varepsilon(t)} - \dot \lambda(t) \right) \| \nabla \Phi_{\lambda(t)} (x(t)) \|^2 \\
    &+ \ \frac{1}{2} \left( \dot \lambda(t) \varepsilon^2(t) - \dot \varepsilon(t) + \frac{\gamma \beta \dot \varepsilon(t) \sqrt{\varepsilon(t)}}{2} \right) \| x_{\varepsilon(t), \lambda(t)} \|^2 \\
    &+ \ \left( \frac{\beta \varepsilon^2(t) + \dot \varepsilon(t)}{2} - \frac{\gamma \beta \dot \varepsilon(t) \sqrt{\varepsilon(t)}}{4} \right) \| x(t) \|^2 \\
    &- \ \frac{\beta}{2} \| \nabla \varphi_{\varepsilon(t), \lambda(t)} (x(t)) \|^2 \text{ for all } t \geq t_0.
\end{align*}
Let us now consider
\begin{align*}
    &\mu(t) E(t) \ = \ \mu(t) \left( \varphi_{\varepsilon(t), \lambda(t)} (x(t)) - \varphi_{\varepsilon(t), \lambda(t)} (x_{\varepsilon(t), \lambda(t)}) \right) \\
    &+ \ \frac{\mu(t)}{2} \left \| \gamma \sqrt{\varepsilon(t)} \left( x(t) - x_{\varepsilon(t), \lambda(t)} \right) + \dot x(t) + \beta \nabla \Phi_{\lambda(t)} (x(t)) \right \|^2 \\
    = \ &\mu(t) \left( \varphi_{\varepsilon(t), \lambda(t)} (x(t)) - \varphi_{\varepsilon(t), \lambda(t)} (x_{\varepsilon(t), \lambda(t)}) \right) + \frac{\gamma^2 \mu(t) \varepsilon(t)}{2} \| x(t) - x_{\varepsilon(t), \lambda(t)} \|^2 \\
    &+ \frac{\mu(t)}{2} \| \dot x(t) + \beta \nabla \Phi_{\lambda(t)} (x(t)) \|^2 + \gamma \mu(t) \sqrt{\varepsilon(t)} \left \langle x(t) - x_{\varepsilon(t), \lambda(t)}, \dot x(t) + \beta \nabla \Phi_{\lambda(t)} (x(t)) \right \rangle \\
    \leq \ &\mu(t) \left( \varphi_{\varepsilon(t), \lambda(t)} (x(t)) - \varphi_{\varepsilon(t), \lambda(t)} (x_{\varepsilon(t), \lambda(t)}) \right) + \gamma^2 \mu(t) \varepsilon(t) \| x(t) - x_{\varepsilon(t), \lambda(t)} \|^2 \\
    &+ \frac{\mu(t)}{2} \| \dot x(t) + \beta \nabla \Phi_{\lambda(t)} (x(t)) \|^2 + \gamma \mu(t) \sqrt{\varepsilon(t)} \left \langle x(t) - x_{\varepsilon(t), \lambda(t)}, \dot x(t) \right \rangle + \frac{\beta^2 \mu(t)}{2} \| \nabla \Phi_{\lambda(t)} (x(t)) \|^2,
\end{align*}
since
\begin{align*}
    &\gamma \beta \mu(t) \sqrt{\varepsilon(t)} \left \langle x(t) - x_{\varepsilon(t), \lambda(t)}, \nabla \Phi_{\lambda(t)} (x(t)) \right \rangle \ \leq \ \gamma \beta \mu(t) \sqrt{\varepsilon(t)} \| x(t) - x_{\varepsilon(t), \lambda(t)} \| \| \nabla \Phi_{\lambda(t)} (x(t)) \| \\
    \leq \ &\frac{\gamma^2 \mu(t) \varepsilon(t)}{2} \| x(t) - x_{\varepsilon(t), \lambda(t)} \|^2 + \frac{\beta^2 \mu(t)}{2} \| \nabla \Phi_{\lambda(t)} (x(t)) \|^2.
\end{align*}
Therefore, using $\mu(t) = \left( \alpha - \gamma \right) \sqrt{\varepsilon(t)} - \frac{\dot \varepsilon(t)}{2 \varepsilon(t)}$ (the terms with $\langle x(t) - x_{\varepsilon(t), \lambda(t)}, \dot x(t) \rangle$ disappear), we obtain for all $t \geq t_0$
\begin{align*}
    &\dot E(t) + \mu(t) E(t) \ \leq \ \left( \frac{\gamma \beta \dot \varepsilon(t)}{2 \sqrt{\varepsilon(t)}} + (\alpha - 2 \gamma) \sqrt{\varepsilon(t)} - \frac{\dot \varepsilon(t)}{2 \varepsilon(t)} \right) \left( \varphi_{\varepsilon(t), \lambda(t)} (x(t)) - \varphi_{\varepsilon(t), \lambda(t)} (x_{\varepsilon(t), \lambda(t)}) \right) \\
    &+ \ \left( \gamma^2 (\alpha - \gamma) \varepsilon^{\frac{3}{2}}(t) + \frac{\gamma^3 \varepsilon^{\frac{3}{2}}(t)}{2b} - \frac{\gamma \varepsilon^{\frac{3}{2}}(t)}{2} - \frac{\gamma \beta \dot \varepsilon(t) \sqrt{\varepsilon(t)}}{2} \right) \| x(t) - x_{\varepsilon(t), \lambda(t)} \|^2 \\
    &+ \ \left( \frac{\gamma}{a} + \gamma - \alpha \right) \frac{\sqrt{\varepsilon(t)}}{2} \| \dot x(t) \|^2 - \frac{\dot \varepsilon(t) }{4 \varepsilon(t)} \| \dot x(t) + \beta \nabla \Phi_{\lambda(t)} (x(t)) \|^2 \\
    &+ \ \frac{\gamma (2a + b) \sqrt{\varepsilon(t)}}{2} \left \| \frac{d}{dt} x_{\varepsilon(t), \lambda(t)} \right \|^2 \\
    &+ \ \frac{1}{2} \left( \frac{\gamma \beta^2 \sqrt{\varepsilon(t)}}{a} - \beta + 2 \beta^2 (\alpha - \gamma) \sqrt{\varepsilon(t)} - \frac{\beta^2 \dot \varepsilon(t)}{2 \varepsilon(t)} - \dot \lambda(t) \right) \| \nabla \Phi_{\lambda(t)} (x(t)) \|^2 \\
    &+ \ \frac{1}{2} \left( \dot \lambda(t) \varepsilon^2(t) - \dot \varepsilon(t) + \frac{\gamma \beta \dot \varepsilon(t) \sqrt{\varepsilon(t)}}{2} \right) \| x_{\varepsilon(t), \lambda(t)} \|^2 + \left( \frac{\beta \varepsilon^2(t) + \dot \varepsilon(t)}{2} - \frac{\gamma \beta \dot \varepsilon(t) \sqrt{\varepsilon(t)}}{4} \right) \| x(t) \|^2 \\
    &- \ \frac{\beta}{2} \| \nabla \varphi_{\varepsilon(t), \lambda(t)} (x(t)) \|^2.
\end{align*}
Further we have ($\dot \varepsilon(t) \leq 0$ for all $t \geq t_0$)
\[
- \frac{\dot \varepsilon(t) }{4 \varepsilon(t)} \| \dot x(t) + \beta \nabla \Phi_{\lambda(t)} (x(t)) \|^2 \ \leq \ - \frac{\dot \varepsilon(t) }{2 \varepsilon(t)} \| \dot x(t) \|^2 - \frac{\beta^2 \dot \varepsilon(t) }{2 \varepsilon(t)} \| \nabla \Phi_{\lambda(t)} (x(t)) \|^2.
\]
%since
%\[
%2 \langle \dot x(t), \beta \nabla \Phi_{\lambda(t)} (x(t)) \rangle \ \leq \ \| \dot x(t) \|^2 + \| \beta \nabla \Phi_{\lambda(t)} (x(t)) \|^2.
%\]
As we have established earlier by Lemma \ref{L_1} item \ref{ii_0} and \eqref{T_0}
\[
\left \| \frac{d}{dt} x_{\varepsilon(t), \lambda(t)} \right \|^2 \ \leq \ \left( \frac{2 \dot \lambda(t)}{\lambda(t)} - \frac{\dot \varepsilon(t)}{\varepsilon(t)} \right)^2 \| x_{\varepsilon(t), \lambda(t)} \|^2 \ \leq \ \left( \frac{2 \dot \lambda(t)}{\lambda(t)} - \frac{\dot \varepsilon(t)}{\varepsilon(t)} \right)^2 \| x^* \|^2,
\]
and since there exists $t_1 \geq t_0$ such that $\left( \sqrt{\varepsilon(t)} \to 0, \text{ as } t \to +\infty \right)$
\[
\dot \lambda(t) \varepsilon^2(t) + \left( \frac{\gamma \beta \sqrt{\varepsilon(t)}}{2} - 1 \right) \dot \varepsilon(t) \ \geq \ 0 \text{ for all } t \geq t_1,
\]
we deduce for all $t \geq t_1$
\[
\frac{1}{2} \left( \dot \lambda(t) \varepsilon^2(t) - \dot \varepsilon(t) + \frac{\gamma \beta \dot \varepsilon(t) \sqrt{\varepsilon(t)}}{2} \right) \| x_{\varepsilon(t), \lambda(t)} \|^2 \ \leq \ \frac{1}{2} \left( \dot \lambda(t) \varepsilon^2(t) - \dot \varepsilon(t) + \frac{\gamma \beta \dot \varepsilon(t) \sqrt{\varepsilon(t)}}{2} \right) \| x^* \|^2.
\]
Choosing $b = c \gamma$ with $c > 0$ we obtain for all $t \geq t_1$
\begin{align*}
    &\dot E(t) + \mu(t) E(t) \ \leq \ \left( \frac{\gamma \beta \dot \varepsilon(t)}{2 \sqrt{\varepsilon(t)}} + (\alpha - 2 \gamma) \sqrt{\varepsilon(t)} - \frac{\dot \varepsilon(t)}{2 \varepsilon(t)} \right) \left( \varphi_{\varepsilon(t), \lambda(t)} (x(t)) - \varphi_{\varepsilon(t), \lambda(t)} (x_{\varepsilon(t), \lambda(t)}) \right) \\
    &+ \ \left( \gamma^2 (\alpha - \gamma) \varepsilon^{\frac{3}{2}}(t) + \frac{\gamma^2 \varepsilon^{\frac{3}{2}}(t)}{2c} - \frac{\gamma \varepsilon^{\frac{3}{2}}(t)}{2} - \frac{\gamma \beta \dot \varepsilon(t) \sqrt{\varepsilon(t)}}{2} \right) \| x(t) - x_{\varepsilon(t), \lambda(t)} \|^2 \\
    &+ \ \left( \frac{\gamma}{a} + \gamma - \alpha - \frac{\dot \varepsilon(t)}{\varepsilon^{\frac{3}{2}}(t)} \right) \frac{\sqrt{\varepsilon(t)}}{2} \| \dot x(t) \|^2  \\
    &+ \ \frac{1}{2} \left( \frac{\gamma \beta^2 \sqrt{\varepsilon(t)}}{a} - \beta + 2 \beta^2 (\alpha - \gamma) \sqrt{\varepsilon(t)} - \frac{3 \beta^2 \dot \varepsilon(t)}{2 \varepsilon(t)} - \dot \lambda(t) \right) \| \nabla \Phi_{\lambda(t)} (x(t)) \|^2 \\
    &+ \ \left( \frac{\beta \varepsilon^2(t) + \dot \varepsilon(t)}{2} - \frac{\gamma \beta \dot \varepsilon(t) \sqrt{\varepsilon(t)}}{4} \right) \| x(t) \|^2 \\
    &+ \ \left( \frac{1}{2} \left( \dot \lambda(t) \varepsilon^2(t) - \dot \varepsilon(t) + \frac{\gamma \beta \dot \varepsilon(t) \sqrt{\varepsilon(t)}}{2} \right) + \frac{\gamma (2a + c \gamma) \sqrt{\varepsilon(t)}}{2} \left( \frac{2 \dot \lambda(t)}{\lambda(t)} - \frac{\dot \varepsilon(t)}{\varepsilon(t)} \right)^2 \right) \| x^* \|^2 \\
    &- \frac{\beta}{2} \| \nabla \varphi_{\varepsilon(t), \lambda(t)} (x(t)) \|^2.
\end{align*}
Let us investigate the signs of the terms in the inequality above when $t$ is large enough to satisfy what we assumed before \eqref{assumption_0} -- \eqref{assumption_3}. First of all,
\[
(\alpha - 2 \gamma) \sqrt{\varepsilon(t)} - \frac{\dot \varepsilon(t)}{2 \varepsilon(t)} + \frac{\gamma \beta \dot \varepsilon(t)}{2 \sqrt{\varepsilon(t)}} \ = \ \left( \frac{d}{dt} \left( \frac{1}{\sqrt{\varepsilon(t)}} \right) + \alpha - 2\gamma + \frac{\gamma \beta \dot \varepsilon(t)}{2 \varepsilon(t)} \right) \sqrt{\varepsilon(t)} \ \leq \ 0
\]
due to \eqref{assumption_0}. Secondly,
\begin{align*}
    &\gamma^2 (\alpha - \gamma) \varepsilon^{\frac{3}{2}}(t) + \frac{\gamma^2 \varepsilon^{\frac{3}{2}}(t)}{2c} - \frac{\gamma \varepsilon^{\frac{3}{2}}(t)}{2} - \frac{\gamma \beta \dot \varepsilon(t) \sqrt{\varepsilon(t)}}{2}
    = \ \frac{\gamma \varepsilon^{\frac{3}{2}}(t)}{2} \left( 2 \gamma (\alpha - \gamma)  + \frac{\gamma}{c} - 1 \right) - \frac{\gamma \beta \dot \varepsilon(t) \sqrt{\varepsilon(t)}}{2} \\ 
    = \ &\frac{\gamma \sqrt{\varepsilon(t)}}{2} \left( \left( 2 \gamma (\alpha - \gamma)  + \frac{\gamma}{c} - 1 \right) \varepsilon(t) - \beta \dot \varepsilon(t) \right) \ \leq \ 0
\end{align*}
due to \eqref{assumption_1}. Next we have
\[
\frac{\gamma}{a} + \gamma - \alpha - \frac{\dot \varepsilon(t)}{\varepsilon^{\frac{3}{2}}(t)} \ = \ \frac{d}{dt} \left( \frac{1}{\sqrt{\varepsilon(t)}} \right) + \gamma \frac{a + 1}{a} - \alpha \ \leq \ 0
\]
due to \eqref{assumption_0}. Then,
\begin{align*}
    \frac{\gamma \beta^2 \sqrt{\varepsilon(t)}}{a} - \beta + 2 \beta^2 (\alpha - \gamma) \sqrt{\varepsilon(t)} - \frac{3 \beta^2 \dot \varepsilon(t)}{2 \varepsilon(t)} - \dot \lambda(t) \ \leq \ 0
\end{align*}
due to \eqref{assumption_3}. Finally,
\[
\frac{\beta \varepsilon^2(t) + \dot \varepsilon(t)}{2} - \frac{\gamma \beta \dot \varepsilon(t) \sqrt{\varepsilon(t)}}{4} \ \leq \ 0
\]
due to \eqref{assumption_2}, since 
\[
2 \beta \varepsilon^2(t) + \left( 2 - \gamma \beta \sqrt{\varepsilon(t)} \right) \dot \varepsilon(t) \ \leq \ 0.
\]
So, at the end we deduce for all $t \geq t_1$
\begin{equation}\label{final_form}
\begin{split}
    \dot E(t) + \mu(t) E(t) \ \leq \ &\frac{1}{2} \left( \dot \lambda(t) \varepsilon^2(t) - \dot \varepsilon(t) + \frac{\gamma \beta \dot \varepsilon(t) \sqrt{\varepsilon(t)}}{2} + \gamma (2a + c \gamma) \sqrt{\varepsilon(t)} \left( \frac{2 \dot \lambda(t)}{\lambda(t)} - \frac{\dot \varepsilon(t)}{\varepsilon(t)} \right)^2 \right) \| x^* \|^2 \\
    &- \ \frac{\beta}{2} \| \nabla \varphi_{\varepsilon(t), \lambda(t)} (x(t)) \|^2 \ = \ \frac{g(t) \| x^* \|^2}{2} - \frac{\beta}{2} \| \nabla \varphi_{\varepsilon(t), \lambda(t)} (x(t)) \|^2.
\end{split}
\end{equation}

Integrating \eqref{final_form} from $t_1$ to $t$ we obtain
\[
E(t) - E(t_1) + \int_{t_1}^t \mu(s) E(s) ds + \frac{\beta}{2} \int_{t_1}^t \| \nabla \varphi_{\varepsilon(s)} (x(s)) \|^2 ds \ \leq \ \frac{\| x^* \|^2}{2} \int_{t_1}^t g(s) ds
\]
or, neglecting the positive terms,
\[
\frac{\beta}{2} \int_{t_1}^t \| \nabla \varphi_{\varepsilon(s)} (x(s)) \|^2 ds \ \leq \ E(t_1) + \frac{\| x^* \|^2}{2} \int_{t_1}^t g(s) ds.
\]

From \eqref{final_form} we also obtain for all $t \geq t_1$
\[
\dot E(t) + \mu(t) E(t) \ \leq \ \frac{g(t) \| x^* \|^2}{2}.
\]
Multiplying this with $\Gamma(t) = \exp \left( \int_{t_1}^t \mu(s) ds \right)$ and integrating again on $[t_1, t]$ we deduce
\[
E(t) \ \leq \ \frac{\| x^* \|^2}{2 \Gamma(t)} \int_{t_1}^t \Gamma(s) g(s) ds + \frac{\Gamma(t_1) E(t_1)}{\Gamma(t)}.
\]

\end{proof}

Now that we have the estimate for our energy function $E$, we would like to proceed with the main goal of this section, namely,

\begin{theorem}\label{CR}

Let $x: [t_0, +\infty) \longrightarrow H$ be a solution of \eqref{Syst_0}. Then for any $t \geq t_0$
\[
\Phi_{\lambda(t)} (x(t)) - \Phi^* \ \leq \ E(t) + \frac{\varepsilon(t)}{2} \| x^* \|^2,
\]
\[
\Phi \left( \prox\nolimits_{\lambda(t) \Phi}(x(t)) \right) - \Phi^* \ \leq \ E(t) + \frac{\varepsilon(t)}{2} \| x^* \|^2,
\]
\[
\| \prox\nolimits_{\lambda(t) \Phi} (x(t)) - x(t) \|^2 \ \leq \ 2\lambda(t) E(t) + \lambda(t)\varepsilon(t) \| x^* \|^2
\]
and
\[
\| x(t) - x_{\varepsilon(t), \lambda(t)} \|^2 \ \leq \ \frac{2 E(t)}{\varepsilon(t)}
\]
and the trajectory $x(t)$ converges strongly to $x^*$ as soon as $\lim_{t \to +\infty} \frac{E(t)}{\varepsilon(t)} = 0$.

\end{theorem}

\begin{proof}

Consider
\begin{align*}
    \Phi_{\lambda(t)} (x(t)) - \Phi^* \ = \ &\varphi_{\varepsilon(t), \lambda(t)} (x(t)) - \varphi_{\varepsilon(t), \lambda(t)} (x^*) + \frac{\varepsilon(t)}{2} \left( \| x^* \|^2 - \| x(t) \|^2 \right) \\
    = \ &\varphi_{\varepsilon(t), \lambda(t)} (x(t)) - \varphi_{\varepsilon(t), \lambda(t)} (x_{\varepsilon(t), \lambda(t)}) + \varphi_{\varepsilon(t), \lambda(t)} (x_{\varepsilon(t), \lambda(t)}) - \varphi_{\varepsilon(t), \lambda(t)} (x^*) \\
    &+ \ \frac{\varepsilon(t)}{2} \left( \| x^* \|^2 - \| x(t) \|^2 \right) \\
    \leq \ &\varphi_{\varepsilon(t), \lambda(t)} (x(t)) - \varphi_{\varepsilon(t), \lambda(t)} (x_{\varepsilon(t), \lambda(t)}) + \frac{\varepsilon(t)}{2} \left( \| x^* \|^2 - \| x(t) \|^2 \right)
\end{align*}
Using the definition of $E$ we obtain
\[
\Phi_{\lambda(t)} (x(t)) - \Phi^* \ \leq \ E(t) + \frac{\varepsilon(t)}{2} \| x^* \|^2.
\]
By the definition of the proximal mapping
\[
\Phi_{\lambda(t)}(x(t)) - \Phi^* \ = \ \Phi \left( \prox\nolimits_{\lambda(t) \Phi}(x(t)) \right) - \Phi^* + \frac{1}{2\lambda(t)} \| \prox\nolimits_{\lambda(t) \Phi} (x(t)) - x(t) \|^2 \quad \forall t \geq t_0.
\]
Thus,
\[
\Phi \left( \prox\nolimits_{\lambda(t) \Phi}(x(t)) \right) - \Phi^* \ \leq \ E(t) + \frac{\varepsilon(t)}{2} \| x^* \|^2
\]
and
\[
\frac{1}{2\lambda(t)} \| \prox\nolimits_{\lambda(t) \Phi} (x(t)) - x(t) \|^2 \ \leq \ E(t) + \frac{\varepsilon(t)}{2} \| x^* \|^2.
\]
The second result immediately follows from the $\varepsilon(t)$-strong convexity of $\varphi_{\varepsilon(t), \lambda(t)}$:
\[
\varphi_{\varepsilon(t), \lambda(t)} (x(t)) - \varphi_{\varepsilon(t), \lambda(t)} (x_{\varepsilon(t), \lambda(t)}) \geq \frac{\varepsilon(t)}{2} \| x(t) - x_{\varepsilon(t), \lambda(t)} \|^2
\]
and thus
\[
E(t) \geq \frac{\varepsilon(t)}{2} \| x(t) - x_{\varepsilon(t), \lambda(t)} \|^2.
\]
Finally, by $\lim_{t \to +\infty} \varepsilon(t) = 0$ and \eqref{t_0} we deduce the strong convergence of the trajectories to $x^*$ as soon as $\lim_{t \to +\infty} \frac{E(t)}{\varepsilon(t)} = 0$.

\end{proof}

\section{Further analysis for the general parameters choice}

In this section we will show for the most general setting when the results of the Theorem \ref{CR} make sense, namely, when all the quantities on the right-hand side do converge to zero. Let us notice that since
    \begin{align*} 
        &\dot \lambda(t) \varepsilon^2(t) - \dot \varepsilon(t) + \frac{\gamma \beta \dot \varepsilon(t) \sqrt{\varepsilon(t)}}{2} + \gamma (2a + c \gamma) \sqrt{\varepsilon(t)} \left( \frac{2 \dot \lambda(t)}{\lambda(t)} - \frac{\dot \varepsilon(t)}{\varepsilon(t)} \right)^2 \\
        \leq \ &\dot \lambda(t) \varepsilon^2(t) - \dot \varepsilon(t) + \gamma (2a + c \gamma) \sqrt{\varepsilon(t)} \left( \frac{2 \dot \lambda(t)}{\lambda(t)} - \frac{\dot \varepsilon(t)}{\varepsilon(t)} \right)^2,
    \end{align*}
    we can simplify a bit the analysis of this section, due to
    \begin{align*}
        &\int_{t_1}^t \left( \dot \lambda(s) \varepsilon^2(s) - \dot \varepsilon(s) + \frac{\gamma \beta \dot \varepsilon(s) \sqrt{\varepsilon(s)}}{2} + \gamma (2a + c \gamma) \sqrt{\varepsilon(s)} \left( \frac{2 \dot \lambda(s)}{\lambda(s)} - \frac{\dot \varepsilon(s)}{\varepsilon(s)} \right)^2 \right) \Gamma(s) ds \\
        \leq \ &\int_{t_1}^t \left( \dot \lambda(s) \varepsilon^2(s) - \dot \varepsilon(s) + \gamma (2a + c \gamma) \sqrt{\varepsilon(s)} \left( \frac{2 \dot \lambda(s)}{\lambda(s)} - \frac{\dot \varepsilon(s)}{\varepsilon(s)} \right)^2 \right) \Gamma(s) ds.
    \end{align*}

\subsection{The asymptotic behaviour of the function $\Gamma$}

Let us start with the function $\Gamma(t) \ = \ \exp \left( \int_{t_1}^t \mu(s) ds \right)$ for $\mu(t) = -\frac{\dot \varepsilon(t)}{2 \varepsilon(t)} + \left( \alpha - \gamma \right) \sqrt{\varepsilon(t)}$.
\begin{align*}
    \Gamma(t) \ &= \ \exp \left( -\frac{1}{2} \int_{t_1}^t \frac{\dot \varepsilon(s)}{\varepsilon(s)} ds + (\alpha - \gamma) \int_{t_1}^t \sqrt{\varepsilon(s)} ds \right) \ = \   \exp \left( \frac{1}{2} \ln \frac{\varepsilon(t_1)}{\varepsilon(t)} + (\alpha - \gamma) \int_{t_1}^t \sqrt{\varepsilon(s)} ds \right) \\
    &= \ \sqrt{\frac{\varepsilon(t_1)}{\varepsilon(t)}} \exp \left( (\alpha - \gamma) \int_{t_1}^t \sqrt{\varepsilon(s)} ds \right).
\end{align*}
Since $\varepsilon(t)$ is positive for all $t \geq t_1 \geq t_0$, the integral is nonnegative and the whole exponent is lower bounded by $1$. Using the property of Tikhonov function, namely, $\lim_{t \to +\infty} \varepsilon(t) = 0$, we deduce that 
\[
\Gamma(t) \ \geq \ \sqrt{\frac{\varepsilon(t_1)}{\varepsilon(t)}} \to +\infty \text{ as } t \to +\infty.
\]

\subsection{The asymptotic behaviour of the function $E$}

Assume the following:
\begin{equation}\label{Cond_0}
\begin{split}
    &\lim_{t \to +\infty} \dot \lambda(t) \varepsilon^\frac{3}{2}(t) \ = \ 0 \text{ and } \\
    &\lim_{t \to +\infty} \frac{\dot \lambda(t)}{\lambda(t)} \ = \ 0.
\end{split}
\end{equation}
Let us recall the form of the energy function 
\begin{align*}
    E(t) = \varphi_{\varepsilon(t), \lambda(t)} (x(t)) - \varphi_{\varepsilon(t), \lambda(t)} (x_{\varepsilon(t), \lambda(t)}) + \frac{1}{2} \left \| \gamma \sqrt{\varepsilon(t)} \left( x(t) - x_{\varepsilon(t), \lambda(t)} \right) + \dot x(t) + \beta \nabla \Phi_{\lambda(t)} (x(t)) \right \|^2,
\end{align*}
where $\frac{\alpha}{2} \leq \gamma < \alpha$. Let us study the behaviour of the function
\[
\frac{\int_{t_1}^t \left[ \left( \dot \lambda(s) \varepsilon^2(s) - \dot \varepsilon(s) + \frac{\big( 2 \dot \lambda(s) \varepsilon(s) - \lambda(s) \dot \varepsilon(s) \big)^2 (2a + c \gamma) \gamma}{\lambda^2(s) \varepsilon^{\frac{3}{2}}(s)} \right) \exp\left( \int_{t_1}^s \mu(u) du \right) \right] ds}{\exp\left( \int_{t_1}^t \mu(u) du \right)} \ = \ \frac{\int_{t_1}^t \tilde{g}(s) \Gamma(s) ds}{\Gamma(t)},
\]
as $t \to +\infty$. Since $\tilde{g}(t) \Gamma(t) \ \geq \ 0$ for all $t \geq t_1$ so is the integral $\int_{t_1}^t \tilde{g}(s) \Gamma(s) ds$. If there exists a constant such that $0 \ \leq \ \int_{t_1}^t \tilde{g}(s) \Gamma(s) ds \ \leq \ const$, then $E(t)$ goes to zero as $t \to +\infty$ due to the properties of $h$ and Theorem \ref{Main}. Otherwise, we may apply L'Hospital's rule to obtain
\begin{align*}
    &\lim_{t \to +\infty} \frac{\int_{t_1}^t \tilde{g}(s) \Gamma(s) ds}{\Gamma(t)} \ = \ \lim_{t \to +\infty} \frac{\tilde{g}(t)}{\mu(t)} \\
    = \ &\lim_{t \to +\infty} \left( \frac{2 \dot \lambda(t) \varepsilon^3(t) - 2 \varepsilon(t) \dot \varepsilon(t)}{2(\alpha - \gamma) \varepsilon^\frac{3}{2}(t) - \dot \varepsilon(t)} + \frac{2 \varepsilon(t) \big( 2 \dot \lambda(t) \varepsilon(t) - \lambda(t) \dot \varepsilon(t) \big)^2 (a + c \gamma) \gamma}{\lambda^2(t) \varepsilon^{\frac{3}{2}}(t) \left( 2(\alpha - \gamma) \varepsilon^\frac{3}{2}(t) - \dot \varepsilon(t) \right)} \right),
\end{align*}
if the latest exists, which we are going to show now. Consider
\[
\frac{2 \dot \lambda(t) \varepsilon^3(t) - 2 \varepsilon(t) \dot \varepsilon(t)}{2(\alpha - \gamma) \varepsilon^\frac{3}{2}(t) - \dot \varepsilon(t)} \ = \ 2 \frac{\dot \lambda(t) \varepsilon^\frac{3}{2}(t) - \frac{\dot \varepsilon(t)}{\sqrt{\varepsilon(t)}}}{2(\alpha - \gamma) - \frac{\dot \varepsilon(t)}{\varepsilon^\frac{3}{2}(t)}} \ \leq \ \frac{\dot \lambda(t) \varepsilon^\frac{3}{2}(t) - \frac{\dot \varepsilon(t)}{\sqrt{\varepsilon(t)}}}{\alpha - \gamma},
\]
since $- \frac{\dot \varepsilon(t)}{\varepsilon^\frac{3}{2}(t)} \geq 0$. Notice that
\[
\lim_{t \to +\infty} \frac{-\dot \varepsilon(t)}{\sqrt{\varepsilon(t)}} \ = \ 0 \text{ by } \eqref{assumption_0}, \text{ since } 0 \ \leq \ -\frac{\dot \varepsilon(t)}{\sqrt{\varepsilon(t)}} \ \leq \ 2 \left( \alpha - \gamma \frac{a + 1}{a} \right) \varepsilon(t) \to 0, \text{ as } t \to +\infty.
\]
So, by \eqref{Cond_0} we deduce
\[
\lim_{t \to +\infty} \frac{2 \dot \lambda(t) \varepsilon^3(t) - 2 \varepsilon(t) \dot \varepsilon(t)}{2(\alpha - \gamma) \varepsilon^\frac{3}{2}(t) - \dot \varepsilon(t)} \ = \ 0.
\]
Consider now
\begin{align*}
    \frac{2 \big( 2 \dot \lambda(t) \varepsilon(t) - \lambda(t) \dot \varepsilon(t) \big)^2 (2a + c \gamma) \gamma}{\lambda^2(t) \sqrt{\varepsilon(t)} \left( 2(\alpha - \gamma) \varepsilon^\frac{3}{2}(t) - \dot \varepsilon(t) \right)} \ &\leq \ \frac{2 \big( 2 \dot \lambda(t) \varepsilon(t) - \lambda(t) \dot \varepsilon(t) \big)^2 (2a + c \gamma) \gamma}{2(\alpha - \gamma) \lambda^2(t) \varepsilon^2(t)} \\
    &= \ \frac{(2a + c \gamma) \gamma}{\alpha - \gamma} \left( \frac{2 \dot \lambda(t)}{\lambda(t)} - \frac{\dot \varepsilon(t)}{\varepsilon(t)} \right)^2.
\end{align*}
Again, by \eqref{assumption_0} we know that
\[
0 \ \leq \ -\frac{\dot \varepsilon(t)}{\varepsilon(t)} \ \leq \ 2 \left( \alpha - \gamma \frac{a + 1}{a} \right) \sqrt{\varepsilon(t)} \to 0, \text{ as } t \to +\infty.
\]
So, again using \eqref{Cond_0} we deduce that $\lim_{t \to +\infty} E(t) \ = \ 0$.

\subsection{The asymptotic behaviour of the function $\frac{E}{\varepsilon}$}

Let us assume additionally that
\begin{equation}\label{Cond_1}
\begin{split}
    &\lim_{t \to +\infty} \sqrt{\varepsilon(t)} \exp \left( (\alpha - \gamma) \int_{t_1}^t \sqrt{\varepsilon(s)} ds \right) \ = \ +\infty, \\
    &\lim_{t \to +\infty} \frac{\dot \varepsilon(t)}{\varepsilon^\frac{3}{2}(t)} \ = \ 0, \\
    &\lim_{t \to +\infty} \dot \lambda(t) \sqrt{\varepsilon(t)} \ = \ 0 \text{ and } \\
    &\lim_{t \to +\infty} \frac{\dot \lambda(t)}{\lambda(t) \sqrt{\varepsilon(t)}} \ = \ 0.
\end{split}
\end{equation}
In the same spirit let us analyse the asymptotic behaviour of $\frac{E(t)}{\varepsilon(t)}$ as $t \to +\infty$. From Theorem \ref{Main} we know that
\[
\frac{E(t)}{\varepsilon(t)} \leq \frac{\int_{t_1}^t \tilde{g}(s) \Gamma(s) ds}{\varepsilon(t) \Gamma(t)} \| x^* \| + \frac{E(t_1)}{\varepsilon(t) \Gamma(t)}.
\]
By \eqref{Cond_1} we immediately deduce that $\lim_{t \to +\infty} \frac{E(t_1)}{\varepsilon(t) \Gamma(t)} = 0$. For the first term let us use the same technique as in the previous chapter and apply L'Hospital's rule to obtain
\begin{align*}
    &\lim_{t \to +\infty} \frac{\int_{t_1}^t \tilde{g}(s) \Gamma(s) ds}{\varepsilon(t) \Gamma(t)} \ = \ \lim_{t \to +\infty} \frac{\tilde{g}(t) \Gamma(t)}{\dot \varepsilon(t) \Gamma(t) + \mu(t) \varepsilon(t) \Gamma(t)} \ = \ \lim_{t \to +\infty} \frac{\tilde{g}(t)}{\dot \varepsilon(t) + \mu(t) \varepsilon(t)} \\
    = \ &\lim_{t \to +\infty} \left( \dot \lambda(t) \varepsilon^2(t) - \dot \varepsilon(t) + \frac{\big( 2 \dot \lambda(t) \varepsilon(t) - \lambda(t) \dot \varepsilon(t) \big)^2 (2a + c \gamma) \gamma}{\lambda^2(t) \varepsilon^{\frac{3}{2}}(t)} \right) \frac{1}{\dot \varepsilon(t) + \mu(t) \varepsilon(t)} \\
    = \ &\lim_{t \to +\infty} \left( \dot \lambda(t) \varepsilon^2(t) - \dot \varepsilon(t) + \frac{\big( 2 \dot \lambda(t) \varepsilon(t) - \lambda(t) \dot \varepsilon(t) \big)^2 (2a + c \gamma) \gamma}{\lambda^2(t) \varepsilon^{\frac{3}{2}}(t)} \right) \frac{1}{\frac{\dot \varepsilon(t)}{2} + (\alpha - \gamma) \varepsilon^\frac{3}{2}(t)} \\
    = \ &\lim_{t \to +\infty} \left( \frac{\dot \lambda(t) \varepsilon^2(t) - \dot \varepsilon(t)}{\frac{\dot \varepsilon(t)}{2} + (\alpha - \gamma) \varepsilon^\frac{3}{2}(t)} + \frac{\big( 2 \dot \lambda(t) \varepsilon(t) - \lambda(t) \dot \varepsilon(t) \big)^2 (2a + c \gamma) \gamma}{\lambda^2(t) \varepsilon^{\frac{3}{2}}(t) \left( \frac{\dot \varepsilon(t)}{2} + (\alpha - \gamma) \varepsilon^\frac{3}{2}(t) \right)} \right) \\
    = \ &\lim_{t \to +\infty} \frac{\dot \lambda(t) \lambda^2(t) \varepsilon^{\frac{7}{2}}(t) - \dot \varepsilon(t) \lambda^2(t) \varepsilon^{\frac{3}{2}}(t) + \big( 2 \dot \lambda(t) \varepsilon(t) - \lambda(t) \dot \varepsilon(t) \big)^2 (2a + c \gamma) \gamma}{\frac{\lambda^2(t) \varepsilon^{\frac{3}{2}}(t) \dot \varepsilon(t)}{2} + (\alpha - \gamma) \lambda^2(t) \varepsilon^3(t)} \\
    = \ &\lim_{t \to +\infty} \frac{\dot \lambda(t) \sqrt{\varepsilon(t)} - \frac{\dot \varepsilon(t)}{\varepsilon^{\frac{3}{2}}(t)} + \left( \frac{2 \dot \lambda(t) \varepsilon(t) - \lambda(t) \dot \varepsilon(t)}{\lambda(t) \varepsilon^\frac{3}{2}(t)} \right)^2 (2a + c \gamma) \gamma}{\frac{\dot \varepsilon(t)}{2 \varepsilon^{\frac{3}{2}}(t) } + \alpha - \gamma} \ = \ 0 \text{ by } \eqref{Cond_1}.
\end{align*}
Thus, we have established that $\lim_{t \to +\infty} \frac{E(t)}{\varepsilon(t)} \ = \ 0$.

\subsection{The asymptotic behaviour of the function $\lambda E$}

In this section we need to assume the full set of conditions \eqref{Cond_1} again. We will study the behaviour of $\lambda(t) E(t)$, as $t \to +\infty$. Again, from Theorem \ref{Main} we know that
\[
\lambda(t) E(t) \leq \frac{\lambda(t) \int_{t_1}^t \tilde{g}(s) \Gamma(s) ds}{\Gamma(t)} \| x^* \|  + \frac{E(t_1) \lambda(t)}{\Gamma(t)}.
\]
We immediately obtain that $\frac{E(t_1) \lambda(t)}{\Gamma(t)} \to 0$ as $t \to +\infty$, since
\[
\lim_{t \to +\infty} \frac{\exp \left( (\alpha - \gamma) \int_{t_1}^t \sqrt{\varepsilon(s)} ds \right)}{\lambda(t) \sqrt{\varepsilon(t)}} \ = \ \lim_{t \to +\infty} \frac{\sqrt{\varepsilon(t)} \exp \left( (\alpha - \gamma) \int_{t_1}^t \sqrt{\varepsilon(s)} ds \right)}{\lambda(t) \varepsilon(t)} \ = \ +\infty \text{ by } \eqref{A_0} \text{ and } \eqref{Cond_1}.
\]
Arguing in the same way we deduce for the first term
\[
\lim_{t \to +\infty} \frac{\int_{t_1}^t \tilde{g}(s) \Gamma(s) ds}{\frac{\Gamma(t)}{\lambda(t)}} \ = \ \lim_{t \to +\infty} \frac{\tilde{g}(t) \Gamma(t)}{\frac{\mu(t) \Gamma(t) \lambda(t) - \Gamma(t) \dot \lambda(t)}{\lambda^2(t)}} \ = \ \lim_{t \to +\infty} \frac{\tilde{g}(t) \lambda^2(t)}{\mu(t) \lambda(t) - \dot \lambda(t)}.
\]
Consider
\begin{align*}
    &\frac{\tilde{g}(t) \lambda^2(t)}{\mu(t) \lambda(t) - \dot \lambda(t)} \ = \ \frac{\left( \dot \lambda(t) \varepsilon^2(t) - \dot \varepsilon(t) + \frac{\big( 2 \dot \lambda(t) \varepsilon(t) - \lambda(t) \dot \varepsilon(t) \big)^2 (2a + c \gamma) \gamma}{\lambda^2(t) \varepsilon^{\frac{3}{2}}(t)} \right) \lambda^2(t)}{-\frac{\dot \varepsilon(t) \lambda(t)}{2 \varepsilon(t)} + \left( \alpha - \gamma \right) \sqrt{\varepsilon(t)} \lambda(t) - \dot \lambda(t)} \\
    = \ &2\frac{\dot \lambda(t) \varepsilon^3(t) \sqrt{\varepsilon(t)} \lambda^2(t) - \dot \varepsilon(t) \varepsilon(t) \sqrt{\varepsilon(t)} \lambda^2(t) + \big( 2 \dot \lambda(t) \varepsilon(t) - \lambda(t) \dot \varepsilon(t) \big)^2 (2a + c \gamma) \gamma}{-\dot \varepsilon(t) \sqrt{\varepsilon(t)} \lambda(t) + 2 \left( \alpha - \gamma \right) \varepsilon^2(t) \lambda(t) - 2 \dot \lambda(t) \varepsilon(t) \sqrt{\varepsilon(t)}} \\
    = \ &2\frac{\dot \lambda(t) \varepsilon^\frac{3}{2}(t) \lambda(t) - \frac{\dot \varepsilon(t) \lambda(t)}{\sqrt{\varepsilon(t)}} + \left( \frac{2 \dot \lambda(t) \varepsilon(t) - \lambda(t) \dot \varepsilon(t)}{\varepsilon(t) \sqrt{\lambda(t)}} \right)^2 (2a + c \gamma) \gamma}{\frac{-\dot \varepsilon(t)}{\varepsilon^\frac{3}{2}(t)} + 2 \left( \alpha - \gamma \right) - \frac{2 \dot \lambda(t)}{\lambda(t) \sqrt{\varepsilon(t)}}} \to 0 \text{ as } t \to +\infty \text{ by } \eqref{A_0}, \eqref{assumption_0} \text{ and } \eqref{Cond_1},
\end{align*}
since
\[
0 \ \leq \ \dot \lambda(t) \varepsilon^\frac{3}{2}(t) \lambda(t) \ = \ \dot \lambda(t) \sqrt{\varepsilon(t)} \varepsilon(t) \lambda(t) \to 0, \text{ as } t \to +\infty,
\]
\[
0 \ \leq \ \frac{-\dot \varepsilon(t) \lambda(t)}{\sqrt{\varepsilon(t)}} \ \leq \ 2 \left( \alpha - \gamma \frac{a + 1}{a} \right) \lambda(t) \varepsilon(t) \to 0, \text{ as } t \to +\infty,
\]
\[
0 \ \leq \ \frac{-\dot \varepsilon(t) \sqrt{\lambda(t)}}{\varepsilon(t)} \ \leq \ 2 \left( \alpha - \gamma \frac{a + 1}{a} \right) \sqrt{\varepsilon(t) \lambda(t)} \to 0, \text{ as } t \to +\infty
\]
and
\[
\frac{\dot \lambda(t)}{\sqrt{\lambda(t)}} \to 0, \text{ as } t \to +\infty.
\]
Thus, we have established that $\lim_{t \to +\infty} \frac{E(t)}{\lambda(t)} \ = \ 0$.

\section{Analysis of the conditions}

Let us gather in this section all the conditions that we made in our analysis and show that they all could be satisfied at the same time at least for the polynomial choice of parameters.
\begin{enumerate}[label=(\roman*)]
    
    \item $ \lim_{t \to +\infty} \lambda(t) \varepsilon(t) \ = \ 0 $;

    Suppose that there exist $\frac{\alpha}{2} \leq \gamma < \alpha, \ a > 0 \text{ and } c > 0$ such that for all $t$ large enough it holds that

    \item $ \frac{d}{dt} \left( \frac{1}{\sqrt{\varepsilon(t)}} \right) \ \leq \ \min \left\{ 2 \gamma - \alpha - \frac{\gamma \beta \dot \varepsilon(t)}{2 \varepsilon(t)}, \ \alpha - \gamma \frac{a + 1}{a} \right\} $;

    \item $ \left( 2 \gamma (\alpha - \gamma)  + \frac{\gamma}{c} - 1 \right) \varepsilon(t) - \beta \dot \varepsilon(t) \ \leq \ 0 $;      
    \item $ 2 \beta \varepsilon^2(t) + \left( 2 - \gamma \beta \sqrt{\varepsilon(t)} \right) \dot \varepsilon(t) \ \leq \ 0 $

    and

    \item $ \left( \frac{\gamma}{a} + 2 (\alpha - \gamma) \right) \beta^2 \sqrt{\varepsilon(t)} - \frac{3 \beta^2 \dot \varepsilon(t)}{2 \varepsilon(t)} - \dot \lambda(t) \ \leq \ \beta $.
    
\end{enumerate}

\subsection{Polynomial choice of parameters}

Let us take $\lambda(t) = t^l$ and $\varepsilon(t) = \frac{1}{t^d}$, $l \geq 0$ and $d > 0$. The set of the conditions in this case becomes

\begin{enumerate}[label=(\roman*)]
    
    \item $ \lim_{t \to +\infty} t^{l-d} \ = \ 0 $;

    There exist $\frac{\alpha}{2} \leq \gamma < \alpha, \ a > 0 \text{ and } c > 0$ such that for all $t$ large enough

    \item $ \frac{d}{2} t^{\frac{d}{2}-1} \ \leq \ \min \left\{ 2 \gamma - \alpha + \frac{\gamma \beta d}{2t}, \ \alpha - \gamma \frac{a + 1}{a} \right\} $;

    \item $ \left( 2 \gamma (\alpha - \gamma)  + \frac{\gamma}{c} - 1 \right) \frac{1}{t^d} + \frac{d \beta}{t^{d+1}} \ \leq \ 0 $; 
    
    \item $ \frac{2 \beta}{t^{2d}} - d \left( 2 - \frac{\gamma \beta}{t^\frac{d}{2}} \right) \frac{1}{t^{d+1}} \ \leq \ 0 $

    and

    \item $ \left( \frac{\gamma}{a} + 2 (\alpha - \gamma) \right) \beta^2 \frac{1}{t^\frac{d}{2}} + \frac{3d \beta^2}{2 t} - l t^{l-1} \ \leq \ \beta $.
    
\end{enumerate}

The conditions above are, in turn, equivalent to

\begin{enumerate}[label=(\roman*)]
    
    \item $ l < d $;

    \item $d \leq 2$;

    \item\label{3} $ 2 \gamma (\alpha - \gamma)  \ < \ 1 $; 
    
    \item $d \geq 1$

    and

    \item is always satisfied starting from $t$ large enough.
    
\end{enumerate}

Finally, we deduce for $l$ and $d$
\begin{align*}
    &1 \leq d \leq 2, \\
    &0 \leq l < d.
\end{align*}

\begin{remark}
    Condition \ref{3} does not contradict with the choice of $\gamma$, namely, $\gamma$ could be chosen to satisfy both of them at the same time:
    \[
    2 \gamma (\alpha - \gamma)  \ < \ 1 \text{ and } \frac{\alpha}{2} \leq \gamma < \alpha.
    \]
    Indeed, \ref{3} implies
    \[
    \gamma^2 - \alpha \gamma + \frac{1}{2} > 0.
    \]
    If $\alpha < \sqrt{2}$, then $\gamma^2 - \alpha \gamma + \frac{1}{2}$ is always positive and we are free to choose $\gamma$ such that $\frac{\alpha}{2} \leq \gamma < \alpha$. Otherwise, $\alpha \geq \sqrt{2}$ means that
    \[
    \gamma \ < \ \frac{\alpha - \sqrt{\alpha^2 - 2}}{2} \text{ or } \gamma \ > \ \frac{\alpha + \sqrt{\alpha^2 - 2}}{2}
    \]
    and thus we take
    \[
    \frac{\alpha}{2} \ \leq \ \frac{\alpha + \sqrt{\alpha^2 - 2}}{2} \ < \ \gamma \ < \ \alpha.
    \]
\end{remark}

\section{The precise rates of convergence of the values and trajectories}

Previously, in Theorem \ref{Main}, we have obtained
\[
E(t) \ \leq \ \frac{\| x^* \|^2}{2 \Gamma(t)} \int_{t_1}^t \Gamma(s) g(s) ds + \frac{\Gamma(t_1) E(t_1)}{\Gamma(t)},
\]
where $g(t) = \dot \lambda(t) \varepsilon^2(t) - \dot \varepsilon(t) + \frac{\gamma \beta \dot \varepsilon(t) \sqrt{\varepsilon(t)}}{2} + \gamma (2a + c \gamma) \sqrt{\varepsilon(t)} \left( \frac{2 \dot \lambda(t)}{\lambda(t)} - \frac{\dot \varepsilon(t)}{\varepsilon(t)} \right)^2$, $\Gamma(t) = \exp \left( \int_{t_1}^t \mu(s) ds \right)$ and $\mu(t) = \left( \alpha - \gamma \right) \sqrt{\varepsilon(t)} - \frac{\dot \varepsilon(t)}{2 \varepsilon(t)}$. Now, let us deduce the actual rates of convergence of the function values and trajectories for the same polynomial choice of parameter functions $\lambda(t) = t^l$ and $\varepsilon(t) = \frac{1}{t^d}$, $l \geq 0$ and $d > 0$.

\subsection{The functions $\mu$ and $\Gamma$}

Let us consider the case when $1 \leq d < 2$. The case when $d = 2$ will be treated separately. The function $\mu$ thus writes as follows
$\mu(t) \ = \ \frac{\alpha - \gamma}{t^{\frac{d}{2}}} + \frac{d}{2t}$. Then,
\begin{align*}
    \Gamma(t) \ &= \ \exp \left( \int_{t_1}^t \left[ \frac{\alpha - \gamma}{s^{\frac{d}{2}}} + \frac{d}{2s} \right] ds \right) \ = \ \left( \frac{t}{t_1} \right)^\frac{d}{2} \exp \left( \int_{t_1}^t \frac{\alpha - \gamma}{s^{\frac{d}{2}}} ds \right) \\
    &= \ \left( \frac{t}{t_1} \right)^\frac{d}{2} \exp \left( \frac{\alpha - \gamma}{1 - \frac{d}{2}} \left[ t^{1 - \frac{d}{2}} - t_1^{1 - \frac{d}{2}} \right] \right) \ = \ C t^\frac{d}{2} \exp \left( \frac{\alpha - \gamma}{1 - \frac{d}{2}} t^{1 - \frac{d}{2}} \right),
\end{align*}
where $C = \left( t_1^\frac{d}{2} \exp \left[ \frac{\alpha - \gamma}{1 - \frac{d}{2}}  t_1^{1 - \frac{d}{2}} \right] \right)^{-1}$. So, $\frac{\Gamma(t_1) E(t_1)}{\Gamma(t)}$ goes to zero exponentially, as time goes to infinity due to $1 \leq d < 2$.

\subsection{The function $g$}

First notice that
\begin{align*}
    g(t) \ &= \ \frac{l t^{l-1}}{t^{2d}} + \frac{d}{t^{d+1}} - \frac{\gamma \beta d}{2 t^{d + 1 + \frac{d}{2}}} + \frac{\gamma (2a + c \gamma)}{t^\frac{d}{2}} \left( \frac{2l}{t} + \frac{d}{t} \right)^2 \\
    &= \ l t^{l - 1 - 2d} + \frac{d}{t^{d+1}} - \frac{\gamma \beta d}{2 t^{d + 1 + \frac{d}{2}}} + \frac{\gamma (2a + c \gamma) (2l + d)^2}{t^{\frac{d}{2} + 2}} \\
    &= \ \frac{l}{t^{2d-l+1}} + \frac{d}{t^{d+1}} - \frac{\gamma \beta d}{2 t^{d + 1 + \frac{d}{2}}} + \frac{C_1}{t^{\frac{d}{2} + 2}},
\end{align*}
where $C_1 = \gamma (2a + c \gamma) (2l + d)^2$. Then,
\[
\Gamma(t) g(t) \ = \ C \left( \frac{l}{t^{\frac{3d}{2} + 1 - l}} + \frac{d}{t^{\frac{d}{2}+1}} - \frac{\gamma \beta d}{2 t^{d+1}} + \frac{C_1}{t^2} \right) \exp \left( \frac{\alpha - \gamma}{1 - \frac{d}{2}} t^{1 - \frac{d}{2}} \right).
\]
Let us notice that the behaviour of $\frac{l}{t^{\frac{3d}{2} + 1 - l}} + \frac{d}{t^{\frac{d}{2}+1}} - \frac{\gamma \beta d}{2 t^{d+1}} + \frac{C_1}{t^2}$ is dictated by the term $\frac{1}{t^{\frac{d}{2}+1}}$, as $t \to +\infty$, since, as we have established earlier, $1 \leq d < 2$ and $0 \leq l < d$.

\subsection{Integrating the product $\Gamma g$}

The technique, which will be used in this section, is inspired by \cite{ABCR_0}. First of all, notice that for some $\delta > 0$
\[
\frac{d}{dt} \left( \frac{\exp \left( \frac{\alpha - \gamma}{1 - \frac{d}{2}} t^{1 - \frac{d}{2}} \right)}{\delta t} \right) \ = \ \left( -\frac{1}{\delta t^2} + \frac{\alpha - \gamma}{\delta t^{\frac{d}{2}+1}} \right) \exp \left( \frac{\alpha - \gamma}{1 - \frac{d}{2}} t^{1 - \frac{d}{2}} \right).
\]
Secondly, there exists such $\delta$ that starting from some $t_2 \geq t_1$ it holds that
\[
\frac{l}{t^{\frac{3d}{2} + 1 - l}} + \frac{d}{t^{\frac{d}{2}+1}} - \frac{\gamma \beta d}{2 t^{d+1}} + \frac{C_1}{t^2} \ \leq \ -\frac{1}{\delta t^2} + \frac{\alpha - \gamma}{\delta t^{\frac{d}{2}+1}}.
\]
Thus,
\begin{align*}
    &C \int_{t_2}^t \left( \frac{l}{s^{\frac{3d}{2} + 1 - l}} + \frac{d}{s^{\frac{d}{2}+1}} - \frac{\gamma \beta d}{2 s^{d+1}} + \frac{C_1}{s^2} \right) \exp \left( \frac{\alpha - \gamma}{1 - \frac{d}{2}} s^{1 - \frac{d}{2}} \right) ds \\
    \leq \ &C \int_{t_2}^t \left( -\frac{1}{\delta s^2} + \frac{\alpha - \gamma}{\delta s^{\frac{d}{2}+1}} \right) \exp \left( \frac{\alpha - \gamma}{1 - \frac{d}{2}} s^{1 - \frac{d}{2}} \right) ds \\
    = \ &C \int_{t_2}^t \frac{d}{ds} \left( \frac{\exp \left( \frac{\alpha - \gamma}{1 - \frac{d}{2}} s^{1 - \frac{d}{2}} \right)}{\delta s} \right) ds \ = \ C \left( \frac{\exp \left( \frac{\alpha - \gamma}{1 - \frac{d}{2}} t^{1 - \frac{d}{2}} \right)}{\delta t} - \frac{\exp \left( \frac{\alpha - \gamma}{1 - \frac{d}{2}} t_2^{1 - \frac{d}{2}} \right)}{\delta t_2} \right) \\
    = \ &C \frac{\exp \left( \frac{\alpha - \gamma}{1 - \frac{d}{2}} t^{1 - \frac{d}{2}} \right)}{\delta t} - C_2,
\end{align*}
where $C_2 = C \frac{\exp \left( \frac{\alpha - \gamma}{1 - \frac{d}{2}} t_2^{1 - \frac{d}{2}} \right)}{\delta t_2}$.

\subsection{Finalizing the estimates}

Let us return to
\begin{align*}
    E(t) \ \leq \ &\frac{\| x^* \|^2}{2 \Gamma(t)} \int_{t_1}^t \Gamma(s) g(s) ds + \frac{\Gamma(t_1) E(t_1)}{\Gamma(t)} \ = \ \frac{\| x^* \|^2}{2 \Gamma(t)} \int_{t_1}^{t_2} \Gamma(s) g(s) ds + \frac{\| x^* \|^2}{2 \Gamma(t)} \int_{t_2}^t \Gamma(s) g(s) ds \\
    &+ \ \frac{\Gamma(t_1) E(t_1)}{\Gamma(t)} \ \leq \ \frac{\| x^* \|^2}{2} \left( \frac{1}{\Gamma(t)} \int_{t_1}^{t_2} \Gamma(s) g(s) ds - \frac{C_2}{\Gamma(t)} + C \frac{\exp \left( \frac{\alpha - \gamma}{1 - \frac{d}{2}} t^{1 - \frac{d}{2}} \right)}{\delta t \Gamma(t)} \right) + \frac{\Gamma(t_1) E(t_1)}{\Gamma(t)}.
\end{align*}
This expression converges to zero at a speed of the slowest decaying term (all the other decay exponentially):
\[
C \frac{\exp \left( \frac{\alpha - \gamma}{1 - \frac{d}{2}} t^{1 - \frac{d}{2}} \right)}{\delta t \Gamma(t)} \ = \ C \frac{\exp \left( \frac{\alpha - \gamma}{1 - \frac{d}{2}} t^{1 - \frac{d}{2}} \right)}{\delta C t^{\frac{d}{2}+1} \exp \left( \frac{\alpha - \gamma}{1 - \frac{d}{2}} t^{1 - \frac{d}{2}} \right)} \ = \ \frac{1}{\delta t^{\frac{d}{2}+1}}.
\]
Thus, there exists a constant $C_3 > 0$ such that for all $t \geq t_2$
\[
E(t) \ \leq \ \frac{C_3}{t^{\frac{d}{2}+1}}.
\]

\subsection{The rates themselves}

Now we can deduce the actual rates for the quantities in Theorem \ref{CR}. For all $t \geq t_2$
\[
\Phi_{\lambda(t)} (x(t)) - \Phi^* \ \leq \ \frac{C_3}{t^{\frac{d}{2}+1}} + \frac{1}{2t^d} \| x^* \|^2,
\]
\[
\Phi \left( \prox\nolimits_{\lambda(t) \Phi}(x(t)) \right) - \Phi^* \ \leq \ \frac{C_3}{t^{\frac{d}{2}+1}} + \frac{1}{2t^d} \| x^* \|^2,
\]
\[
\| \prox\nolimits_{\lambda(t) \Phi} (x(t)) - x(t) \|^2 \ \leq \ 2 C_3 t^{l - \frac{d}{2} - 1} + t^{l-d} \| x^* \|^2
\]
and
\[
\| x(t) - x_{\varepsilon(t), \lambda(t)} \|^2 \ \leq \ \frac{C_3}{t^{1-\frac{d}{2}}}.
\]
Again, there exist constants $C_4, C_5 > 0$ such that for all $t \geq t_2$
\[
\Phi_{\lambda(t)} (x(t)) - \Phi^* \ \leq \ \frac{C_4}{t^d},
\]
\[
\Phi \left( \prox\nolimits_{\lambda(t) \Phi}(x(t)) \right) - \Phi^* \ \leq \ \frac{C_4}{t^d},
\]
\[
\| \prox\nolimits_{\lambda(t) \Phi} (x(t)) - x(t) \|^2 \ \leq \ \frac{C_5}{t^{\frac{d}{2} + 1 - l}}
\]
and
\[
\| x(t) - x_{\varepsilon(t), \lambda(t)} \|^2 \ \leq \ \frac{C_3}{t^{1-\frac{d}{2}}}.
\]

\section{The rates of convergence of the function values in case $d=2$}

This particular case is of a great interest, as it is in a way a bordering case, when one cannot show the strong convergence of the trajectories, but still can show the fast convergence of the values. In this case the functions $\mu$ and $\Gamma$ are
\[
\mu(t) \ = \ \frac{\alpha - \gamma + 1}{t}
\]
and
\[
\Gamma(t) \ = \ \exp \left( \int_{t_1}^t \frac{\alpha - \gamma + 1}{s} ds \right) \ = \ \left( \frac{t}{t_1} \right)^{\alpha - \gamma + 1} \ = \ C t^{\alpha - \gamma + 1},
\]
where $C = \frac{1}{t_1^{\alpha - \gamma + 1}}$. The function $g$ is
\[
g(t) \ = \ \frac{l}{t^{5-l}} + \frac{2}{t^3} - \frac{\gamma \beta}{t^4} + \frac{C_1}{t^3} \ = \ \frac{l}{t^{5-l}} + \frac{2 + C_1}{t^3} - \frac{\gamma \beta}{t^4},
\]
where $C_1 = 4 \gamma (2a + c \gamma) (l + 1)^2$. Thus,
\[
\Gamma(t) g(t) \ = \ C \left( l t^{\alpha - \gamma + l - 4} + \left( 2 + C_1 \right) t^{\alpha - \gamma - 2} - \gamma \beta t^{\alpha - \gamma - 3} \right).
\]
So, 
\begin{align*}
    &C \int_{t_1}^t \left( l s^{\alpha - \gamma + l - 4} + \left( 2 + C_1 \right) s^{\alpha - \gamma - 2} - \gamma \beta s^{\alpha - \gamma - 3} \right) ds \\
    = \ &C \left( \frac{l s^{\alpha - \gamma + l - 3}}{\alpha - \gamma + l - 3} + \frac{\left( 2 + C_1 \right) s^{\alpha - \gamma - 1}}{\alpha - \gamma - 1} - \frac{\gamma \beta s^{\alpha - \gamma - 2}}{\alpha - \gamma - 2} \right)\Bigg|_{t_1}^t \\
    = \ &C \left( \frac{l t^{\alpha - \gamma + l - 3}}{\alpha - \gamma + l - 3} + \frac{\left( 2 + C_1 \right) t^{\alpha - \gamma - 1}}{\alpha - \gamma - 1} - \frac{\gamma \beta t^{\alpha - \gamma - 2}}{\alpha - \gamma - 2} \right) - C_2,
\end{align*}
where $C_2 = C \left( \frac{l t_1^{\alpha - \gamma + l - 3}}{\alpha - \gamma + l - 3} + \frac{\left( 2 + C_1 \right) t_1^{\alpha - \gamma - 1}}{\alpha - \gamma - 1} - \frac{\gamma \beta t_1^{\alpha - \gamma - 2}}{\alpha - \gamma - 2} \right)$. By Theorem 5 we have
\begin{align*}
    E(t) \ &\leq \ \frac{\| x^* \|^2}{2} \frac{C \left( \frac{l t^{\alpha - \gamma + l - 3}}{\alpha - \gamma + l - 3} + \frac{\left( 2 + C_1 \right) t^{\alpha - \gamma - 1}}{\alpha - \gamma - 1} - \frac{\gamma \beta t^{\alpha - \gamma - 2}}{\alpha - \gamma - 2} \right) - C_2}{C t^{\alpha - \gamma + 1}} + \frac{C t_1^{\alpha - \gamma + 1} E(t_1)}{C t^{\alpha - \gamma + 1}} \\
    &= \ \frac{\| x^* \|^2}{2} \frac{\frac{l t^{\alpha - \gamma + l - 3}}{\alpha - \gamma + l - 3} + \frac{\left( 2 + C_1 \right) t^{\alpha - \gamma - 1}}{\alpha - \gamma - 1} - \frac{\gamma \beta t^{\alpha - \gamma - 2}}{\alpha - \gamma - 2}}{t^{\alpha - \gamma + 1}} + \frac{C_3}{t^{\alpha - \gamma + 1}} \\
    &= \ \frac{\| x^* \|^2}{2} \left( \frac{l t^{l - 4}}{\alpha - \gamma + l - 3} + \frac{\left( 2 + C_1 \right) t^{-2}}{\alpha - \gamma - 1} - \frac{\gamma \beta t^{-3}}{\alpha - \gamma - 2} \right) + \frac{C_3}{t^{\alpha - \gamma + 1}}
\end{align*}
where $C_3 = \frac{2C t_1^{\alpha - \gamma + 1} E(t_1) - C_2 \| x^* \|^2}{2C}$. We know that $\frac{\alpha}{2} \leq \gamma < \alpha$ and $0 \leq l < 2$. Thus, in the brackets the term with $t^{-2}$ is dominating, as $t \to +\infty$. Moreover, $\alpha - \gamma + 1 > 1$. So, the behaviour of the entire expression depends on the value of $\alpha$. There exists a constant $C_4$ such that for all $t \geq t_1$
\[
E(t) \ \leq \ \frac{C_4}{t^2} + \frac{C_3}{t^{\alpha - \gamma + 1}}.
\]
That leads us to the following rates for all $t \geq t_1$
\[
\Phi_{\lambda(t)} (x(t)) - \Phi^* \ \leq \ \frac{C_4}{t^2} + \frac{C_3}{t^{\alpha - \gamma + 1}} + \frac{\| x^* \|^2}{2t^2},
\]
\[
\Phi \left( \prox\nolimits_{\lambda(t) \Phi}(x(t)) \right) - \Phi^* \ \leq \ \frac{C_4}{t^2} + \frac{C_3}{t^{\alpha - \gamma + 1}} + \frac{\| x^* \|^2}{2t^2},
\]
\[
\| \prox\nolimits_{\lambda(t) \Phi} (x(t)) - x(t) \|^2 \ \leq \ \frac{2C_4}{t^{2-l}} + \frac{2C_3}{t^{\alpha - \gamma - l + 1}} + \frac{\| x^* \|^2}{t^{2-l}}
\]
and
\[
\| x(t) - x_{\varepsilon(t), \lambda(t)} \|^2 \ \leq \ 2C_4 + \frac{2C_3}{t^{\alpha - \gamma - 1}}.
\]
As we can see, the strong convergence of the trajectories can no longer be shown. Nevertheless, for $C_5 = \frac{2C_4 + \| x^* \|^2}{2}$ we deduce for all $t \geq t_1$
\[
\Phi_{\lambda(t)} (x(t)) - \Phi^* \ \leq \ \frac{C_5}{t^2} + \frac{C_3}{t^{\alpha - \gamma + 1}},
\]
\[
\Phi \left( \prox\nolimits_{\lambda(t) \Phi}(x(t)) \right) - \Phi^* \ \leq \ \frac{C_5}{t^2} + \frac{C_3}{t^{\alpha - \gamma + 1}} 
\]
and
\[
\| \prox\nolimits_{\lambda(t) \Phi} (x(t)) - x(t) \|^2 \ \leq \ \frac{2C_5}{t^{2-l}} + \frac{2C_3}{t^{\alpha - \gamma - l + 1}}.
\]
Since we are free to choose $\gamma$ such that $\frac{\alpha}{2} \leq \gamma < \alpha$, and since we want to have as fast rates as possible, we should take $\gamma = \frac{\alpha}{2}$.
\[
\Phi_{\lambda(t)} (x(t)) - \Phi^* \ \leq \ \frac{C_5}{t^2} + \frac{C_3}{t^{\frac{\alpha}{2} + 1}},
\]
\[
\Phi \left( \prox\nolimits_{\lambda(t) \Phi}(x(t)) \right) - \Phi^* \ \leq \ \frac{C_5}{t^2} + \frac{C_3}{t^{\frac{\alpha}{2} + 1}} 
\]
and
\[
\| \prox\nolimits_{\lambda(t) \Phi} (x(t)) - x(t) \|^2 \ \leq \ \frac{2C_5}{t^{2-l}} + \frac{2C_3}{t^{\frac{\alpha}{2} - l + 1}}.
\]
Here we have to consider several cases.
\begin{enumerate}

    \item If $0 < \alpha < 2$, then there exists $C_6$ such that for all $t \geq t_1$
    \[
    \Phi_{\lambda(t)} (x(t)) - \Phi^* \ \leq \ \frac{C_6}{t^{\frac{\alpha}{2} + 1}},
    \]
    \[
    \Phi \left( \prox\nolimits_{\lambda(t) \Phi}(x(t)) \right) - \Phi^* \ \leq \ \frac{C_6}{t^{\frac{\alpha}{2} + 1}} 
    \]
    and
    \[
    \| \prox\nolimits_{\lambda(t) \Phi} (x(t)) - x(t) \|^2 \ \leq \ \frac{2C_6}{t^{\frac{\alpha}{2} - l + 1}}.
    \]

    \item If $\alpha \geq 2$, then there exists $C_6$ such that for all $t \geq t_1$
    \[
    \Phi_{\lambda(t)} (x(t)) - \Phi^* \ \leq \ \frac{C_6}{t^2},
    \]
    \[
    \Phi \left( \prox\nolimits_{\lambda(t) \Phi}(x(t)) \right) - \Phi^* \ \leq \ \frac{C_6}{t^2} 
    \]
    and
    \[
    \| \prox\nolimits_{\lambda(t) \Phi} (x(t)) - x(t) \|^2 \ \leq \ \frac{2C_6}{t^{2-l}}.
    \]
    
\end{enumerate}

\begin{remark}

Probably, it is possible to show the weak convergence of the trajectories to a minimizer of the objective function in case $d=2$. 
    
\end{remark}

\section{Numerical examples}

\subsection{The rates of convergence of the Moreau envelope values}

Let us consider the following objective function $\Phi: \mathbb{R} \to \mathbb{R}$, $\Phi(x) = |x| + \frac{x^2}{2}$ and plot the values of its Moreau envelope for different polynomial functions $\lambda$ and $\varepsilon$ in order to illustrate the theoretical results with some numerical examples. We set $\lambda(t) = t^l$ and $\varepsilon(t) = \frac{1}{t^d}$ with $x(t_0) = x_0 = 10$, $\dot x(t_0) = 0$, $\alpha = 10$, $\beta = 1$ and $t_0 = 1$.

Consider different Moreau envelope parameters $\lambda$ with $d = 1.9$:

\begin{figure}[H]
    \centering
    \includegraphics[width=\textwidth]{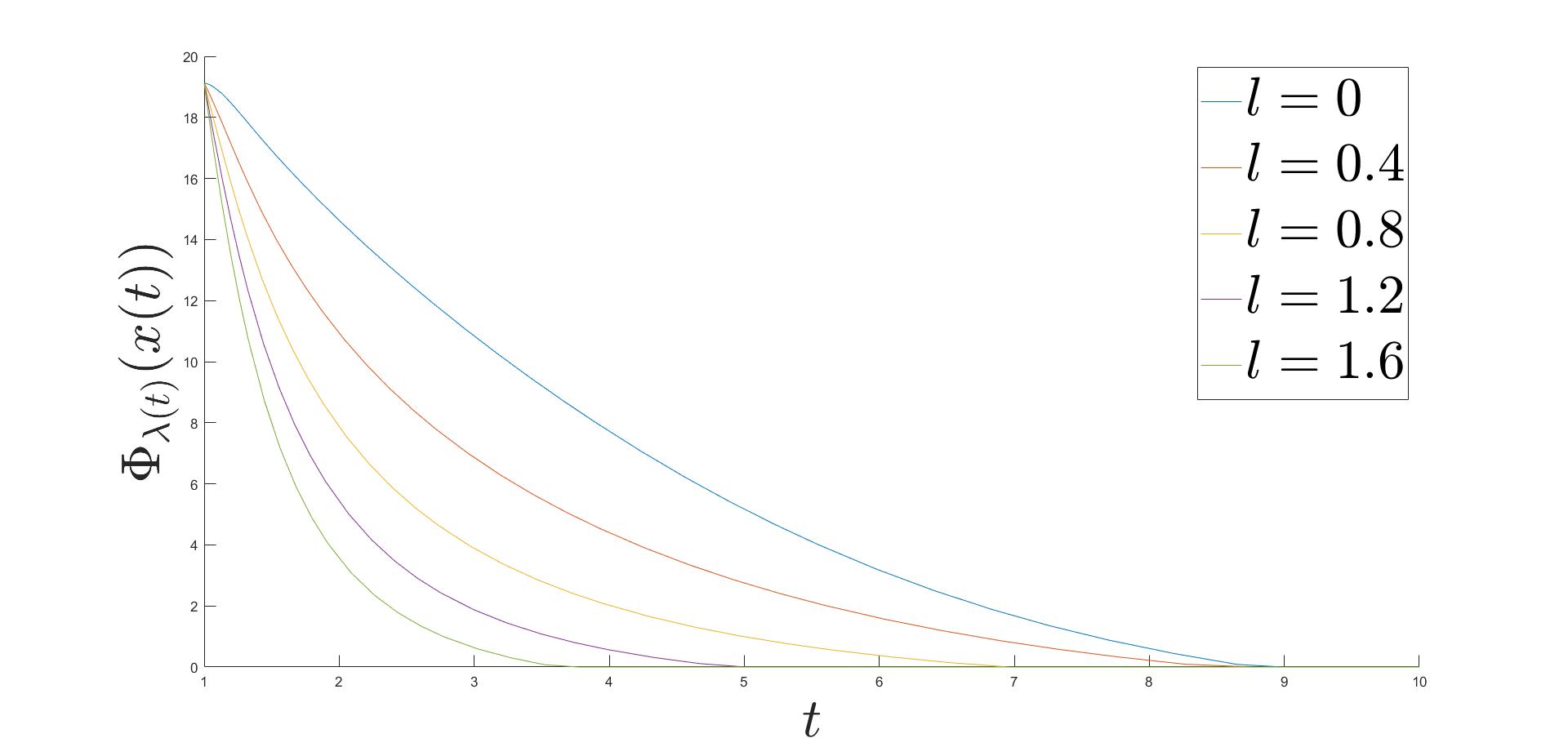}
    \caption{Moreau envelope values for $d = 1.9$}
\end{figure}

We notice that a faster growing function $\lambda$ implies faster convergence of the Moreau envelope of the objective function $\Phi$.

Increasing the speed of decay of the Tikhonov function $\varepsilon$ for a fixed $l = 1$ accelerates the convergence of the Moreau envelope values, which was predicted by the theory:

\begin{figure}[H]
    \centering
    \includegraphics[width=\textwidth]{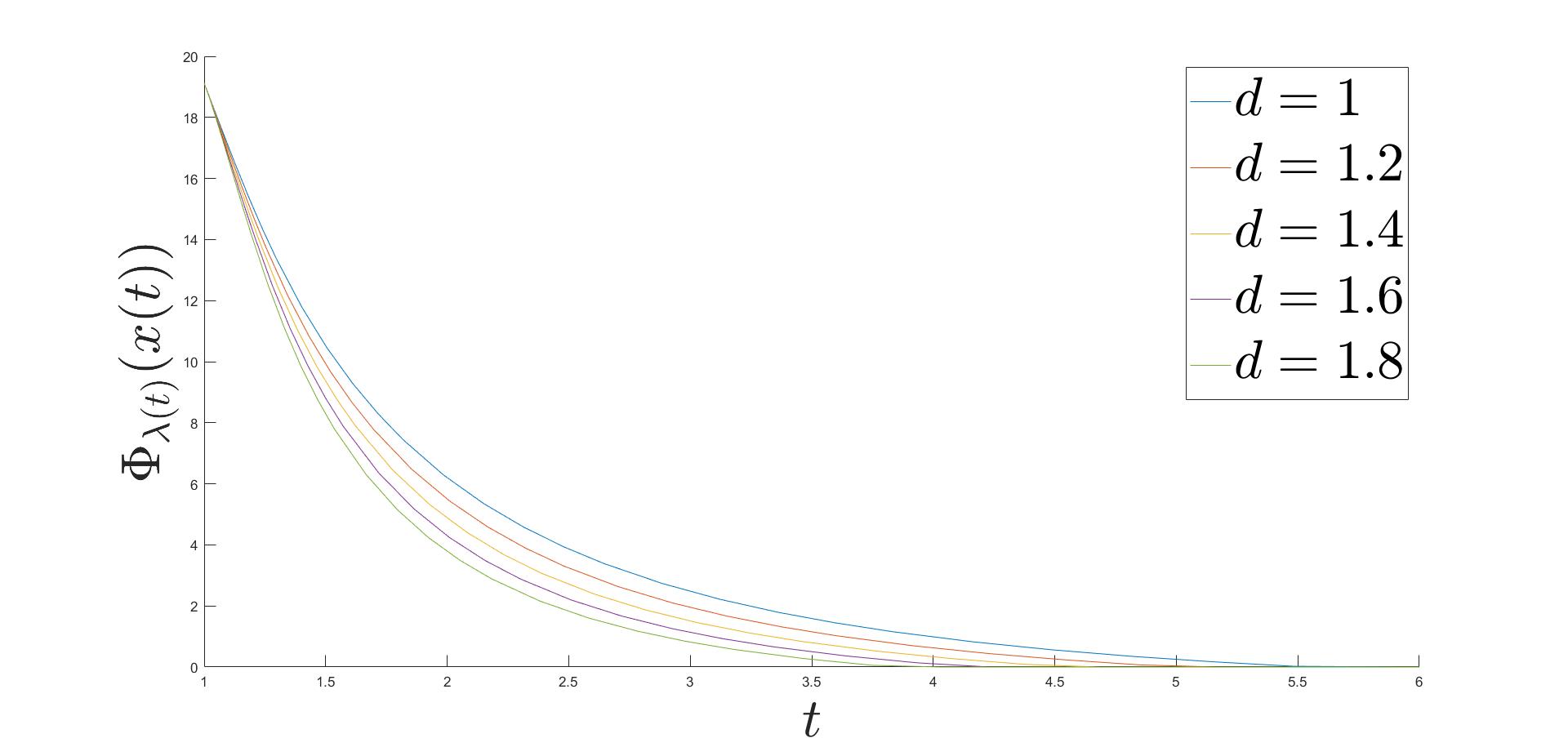}
    \caption{Moreau envelope values for $l = 1$}
\end{figure}

\subsection{Strong convergence of the trajectories}

For the different objective function let us investigate the strong convergence of the trajectories of \eqref{Syst_0} and show some examples when the trajectories actually diverge due to one of the key assumptions of the analysis not being fulfilled. We define
\begin{equation*}
    \Phi(x) \ = \ 
    \begin{cases}
        &|x - 1|, \ x > 1 \\
        &0, \ x \in [-1, 1] \\
        &|x + 1|, \ x < -1.
    \end{cases}
\end{equation*}
The set $\argmin \Phi$ is the segment $[-1, 1]$ and clearly $0$ is its element of the minimal norm. Let us investigate the unfluence of the Tikhonov term on the behaviour of the trajectories of the system for $\lambda(t) = t$.

\begin{figure}[H]
     \centering
     \begin{subfigure}[b]{0.49\textwidth}
         \centering
         \includegraphics[width=\textwidth]{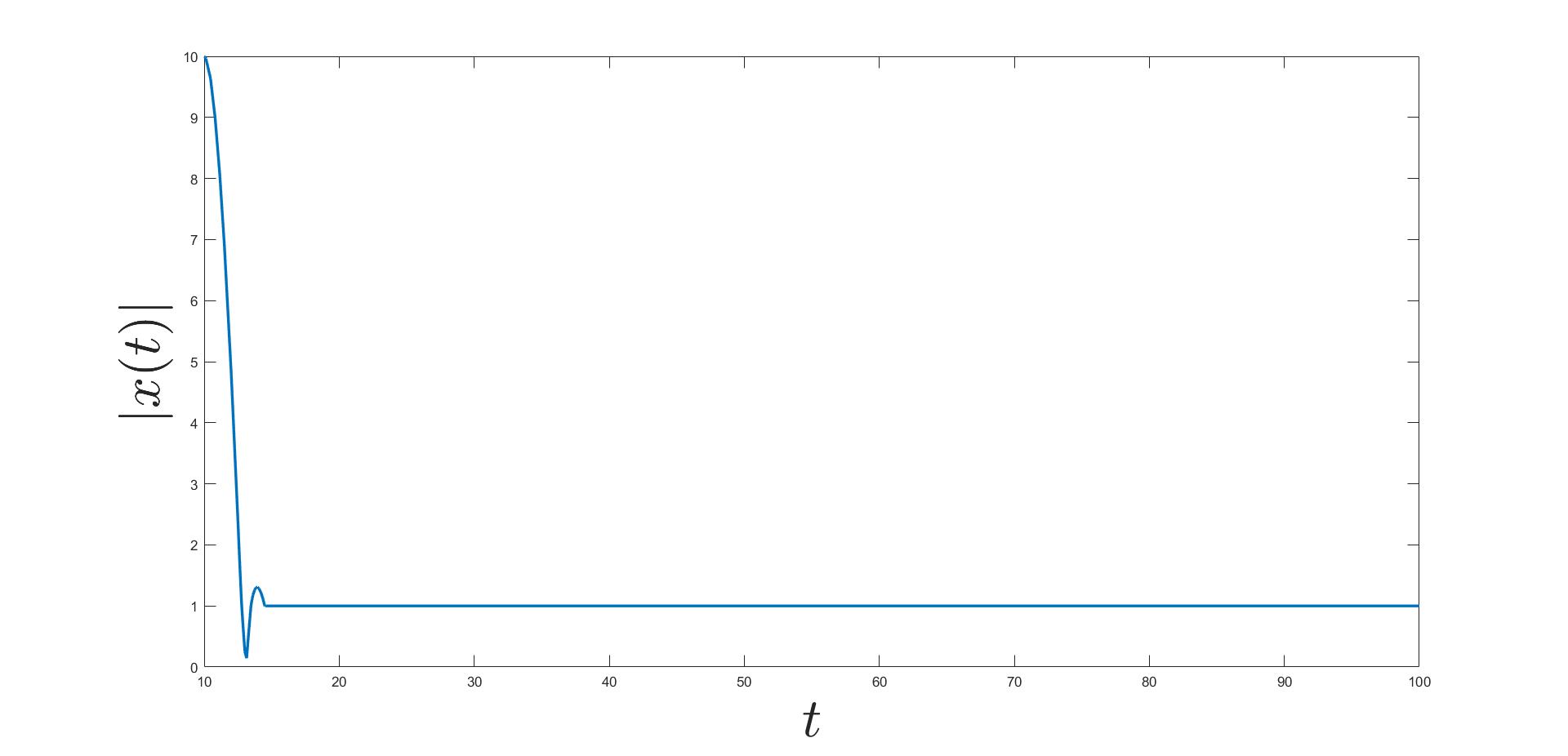}
         \caption{$\varepsilon(t) = 0$}
         %\label{}
     \end{subfigure}
     \hfill
     \begin{subfigure}[b]{0.49\textwidth}
         \centering
         \includegraphics[width=\textwidth]{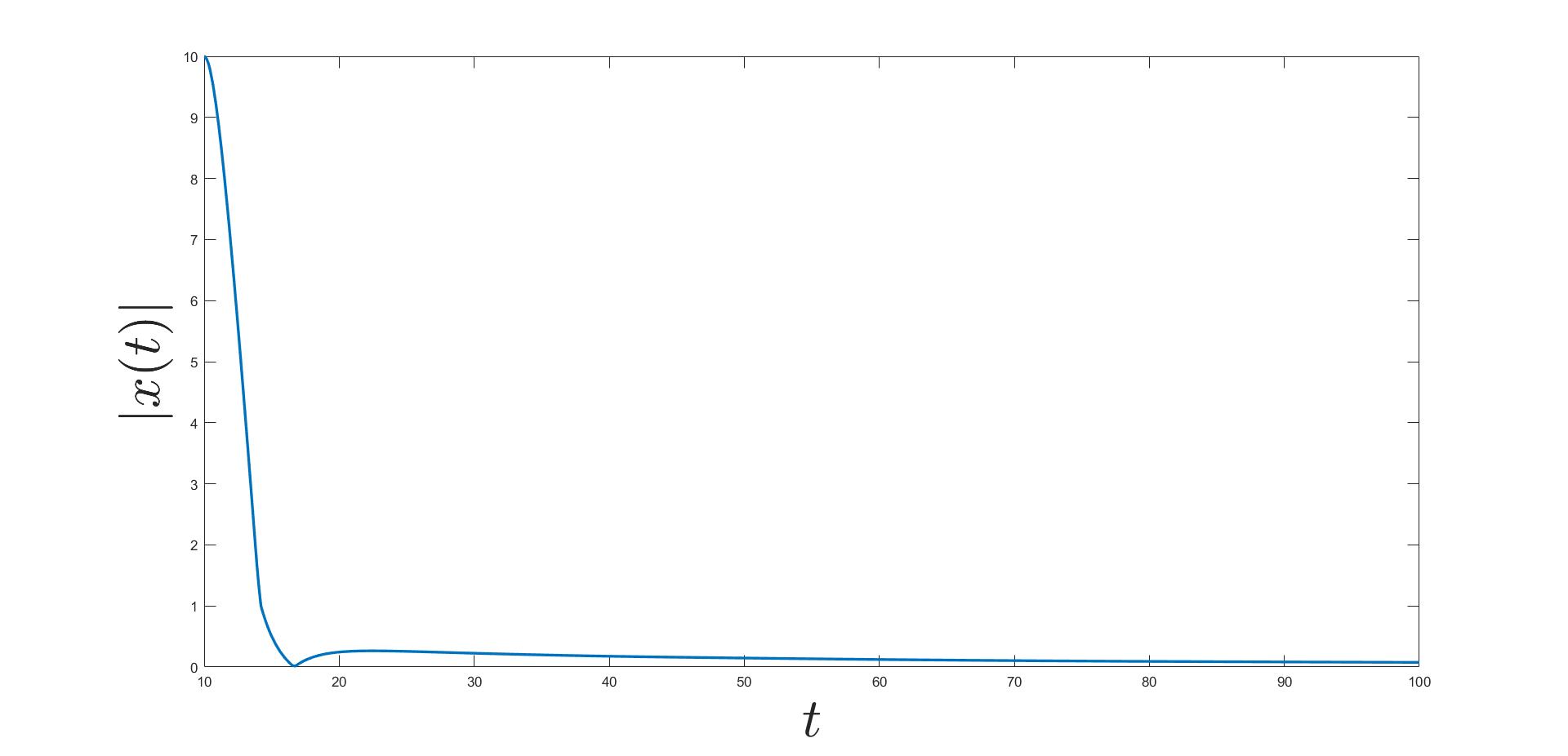}
         \caption{$\varepsilon(t) = \frac{1}{t^\frac{3}{2}}$}
         %\label{}
     \end{subfigure}
     \hfill
     \caption{The role of Tikhonov term.}
\end{figure}

As we can see in case Tikhonov function is missing the trajectories converge to a minimizer $1$ of $\Phi$, however, Tikhonov term ensures the convergence to the minimal norm solution $0$.

Finally, for the same choice of $\lambda$ and $\Phi$ let us take different Tikhonov functions to study their effect on the trajectories of \eqref{Syst_0}. For this purpose we increase the starting point to $x(t_0) = 100$.

\begin{figure}[H]
    \includegraphics[width=\textwidth]{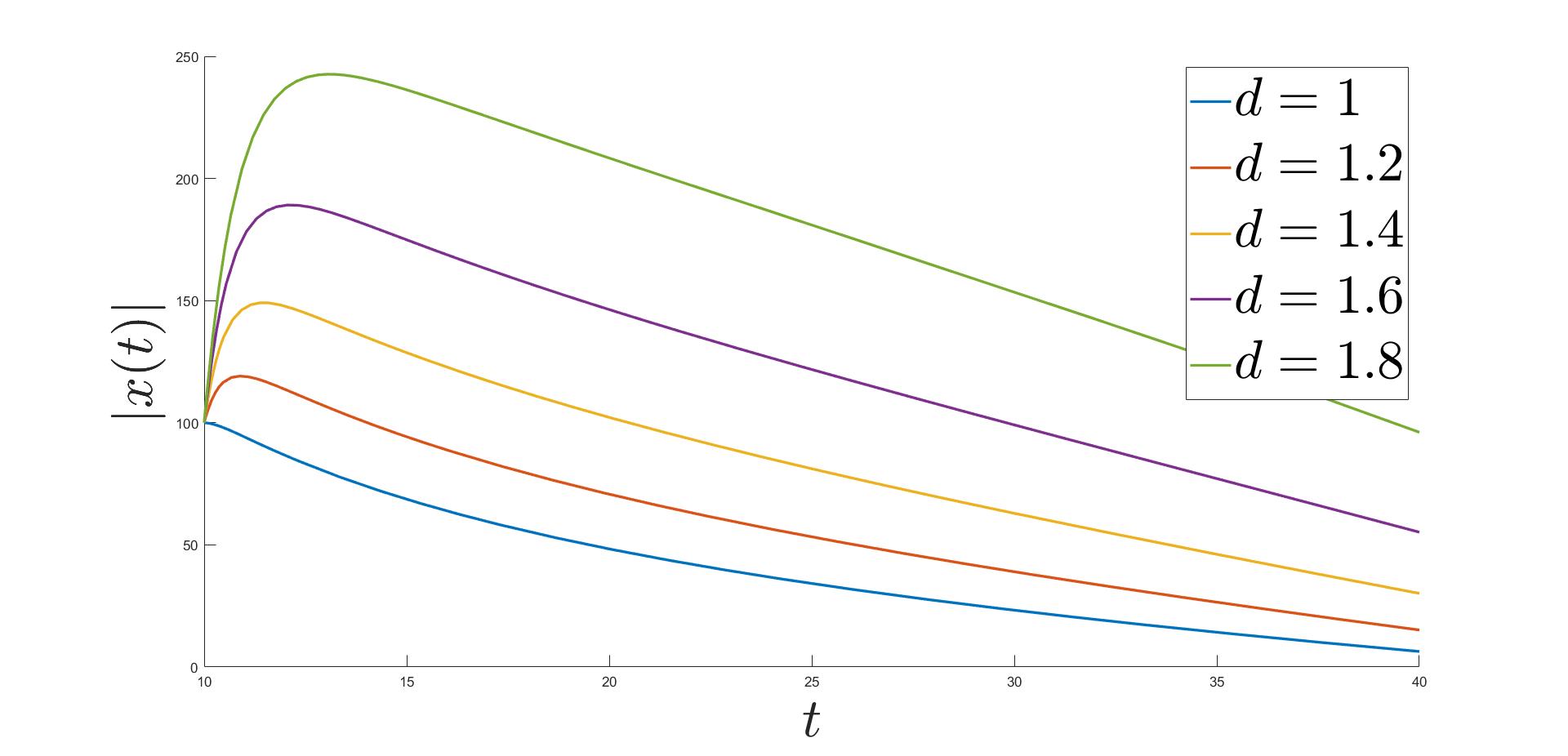}
    \caption{$l = 1$.}
\end{figure}

As we see, the faster $\varepsilon$ decays, the slower trajectories converge, which totally corresponds to the theoretical results.

To end this section let us break some of the fundamental conditions of our analysis and show that there is no convergence of the trajectories in this case.

\begin{figure}[H]
    \centering
    \includegraphics[width=\textwidth]{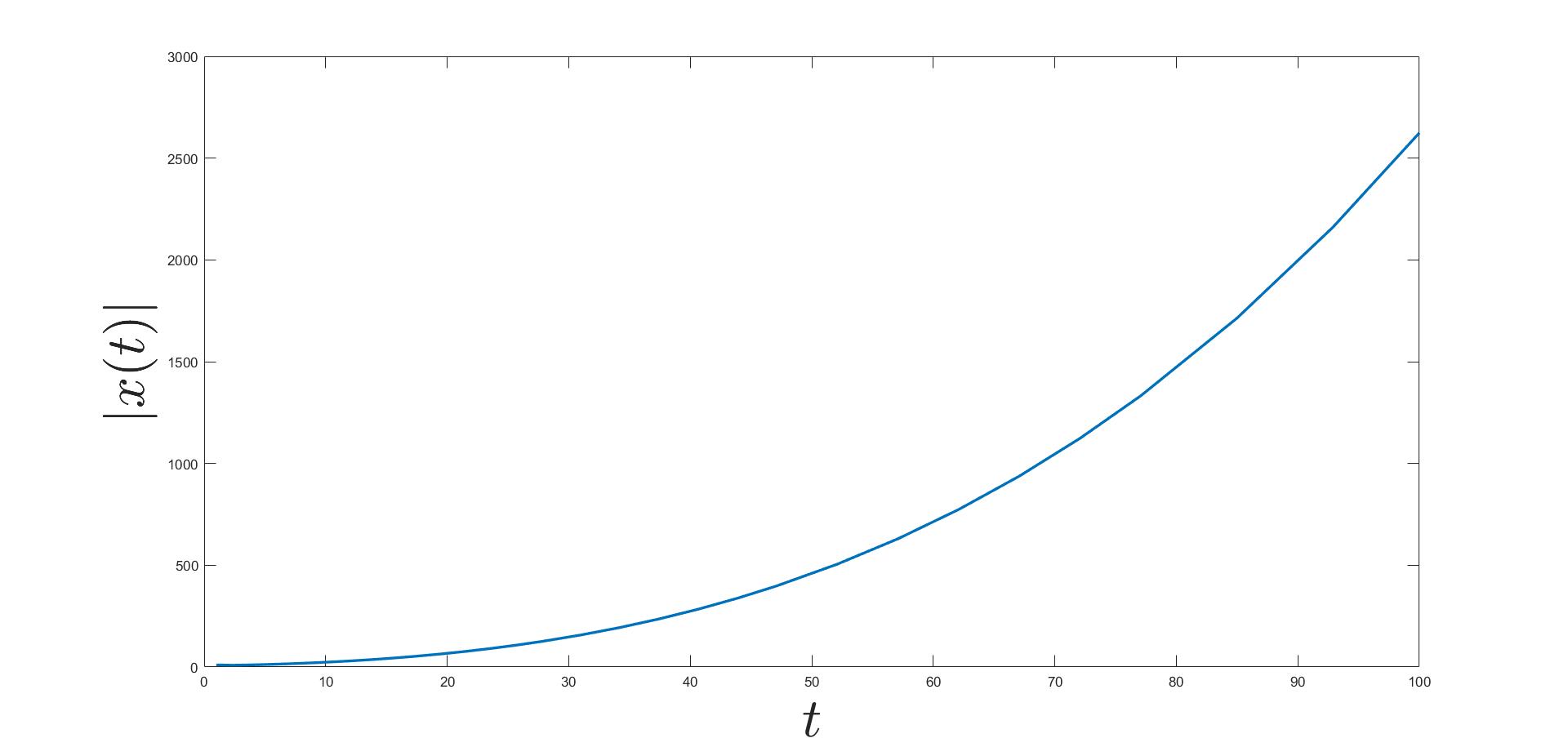}
    \caption{$l$ and $d$ do not meet the requirements.}
\end{figure}

The author is immensely grateful to Professor R.I. Bo\c t for valuable comments and fruitful discussions, which significantly improved the quality of this manuscript.

\section*{Appendix}

\begin{lemma}\label{O}

Let $S$ be a non-empty subset of a real Hilbert space $H$ and $x: [0, +\infty) \mapsto H$ a given map. Assume that
\begin{itemize}

\item for every $z \in S$, $\lim_{t \to +\infty} \| x(t) - z \|$ exists;

\item every weak sequential cluster point of the map $x$ belongs to $S$.

\end{itemize}
Then $x(t)$ converges weakly to some element of $S$ as $t \to +\infty$.

\end{lemma}

\end{document}